\documentclass[11pt]{amsart}

\usepackage[USenglish]{babel}

\usepackage{graphicx}
\usepackage{subcaption}

\usepackage{rank-2-roots}

\usepackage[pagebackref, colorlinks=true, linkcolor=black, citecolor=black, urlcolor=black]{hyperref}

\usepackage{xargs}                      % Use more than one optional parameter in a new commands
\usepackage[colorinlistoftodos,prependcaption %,textsize=tiny
]{todonotes}                                                                                                                 %todo notes!

\newcommandx{\todoN}[2][1=]{\todo[linecolor=red,backgroundcolor=white,bordercolor=red,#1]{#2}}			%Nicola's todo notes in Red
\newcommandx{\todoA}[2][1=]{\todo[linecolor=green,backgroundcolor=white,bordercolor=green,#1]{#2}}			%Allison's todo notes in Green

\makeatletter
\def\namedlabel#1#2{\begingroup
    #2%
    \def\@currentlabel{#2}%
    \phantomsection\label{#1}\endgroup
}

\usepackage{comment}
\usepackage{calc}

\usepackage{amsmath, amsthm, amsfonts, amssymb}
\usepackage{enumerate}
\usepackage[shortlabels]{enumitem}
\usepackage{bbm}
\usepackage{mathrsfs}
\usepackage{amssymb}
\usepackage{mathtools}
\usepackage[all]{xy}
\usepackage{tikz}
\usepackage{tikz-cd}
\usetikzlibrary{matrix,arrows,decorations.pathmorphing}
\usepackage{xcolor}
\usepackage{bm}
	\usetikzlibrary{calc}
	\usetikzlibrary{trees}
	\tikzstyle{every picture}=[scale=.35,inner sep=0]
	\usepackage{color}

\newcommand{\rotsimeq}{\rotatebox[origin=c]{-90}{$\simeq$}}   %for vertical isomorphism

\usepackage{stmaryrd} %%for \llbracket == [[, \rrbracket == ]], \llparenthesis == ((, \rrparenthesis == ))

\usepackage[normalem]{ulem} % for strike through with command \sout{...}

\newtheorem{thm}{Theorem}

\theoremstyle{definition}

\newtheorem{q}{Question}

\theoremstyle{theorem}
\newtheorem{theorem}{Theorem}[section]
\newtheorem{lemma}[theorem]{Lemma}

\theoremstyle{definition}

\theoremstyle{remark}
\newtheorem{remark}[theorem]{Remark}

\numberwithin{equation}{section}

\newcommand{\spinc}{\operatorname{Spin}^c}

\newcommand{\Zdhat}{\widehat{\vphantom{\rule{5pt}{10pt}}\smash{\widehat{Z}}\,}\!}

\makeatletter
\@namedef{subjclassname@2020}{\textup{MSC2020}}
\makeatother

\begin{document}

\title[Gluing and splitting plumbed 3-manifolds]{On gluing and splitting \\ series invariants of plumbed 3-manifolds}

\author[A.H.~Moore]{Allison H.~Moore}
\address{Allison H.~Moore
\newline \indent Department of Mathematics \& Applied Mathematics
\newline \indent Virginia Commonwealth University, Richmond, VA 23284}
\email{moorea14@vcu.edu}

\author[N.~Tarasca]{Nicola Tarasca}
\address{Nicola Tarasca 
\newline \indent Department of Mathematics \& Applied Mathematics
\newline \indent Virginia Commonwealth University, Richmond, VA 23284}
\email{tarascan@vcu.edu}


\begin{abstract}
We study series invariants for plumbed 3-manifolds and knot complements twisted by a root lattice.
Our series recover recent results of Gukov-Pei-Putrov-Vafa, Gukov-Manolescu, Park, and Ri and apply more generally to 3-manifolds which are not necessarily negative definite. We show that our series verify certain gluing and splitting properties related to the corresponding operations on 3-manifolds. We conclude with an explicit description of the case of lens spaces and Brieskorn spheres.
\end{abstract}

\vspace*{-2pc}

\maketitle

\vspace{-2pc}

\section{Introduction}
A new invariant of negative-definite plumbed 3-manifolds has recently been introduced in Gukov-Pei-Putrov-Vafa \cite{gukov2020bps}. It takes the form of a Laurent $q$-series denoted as $\widehat{Z}_a(q)$, with the index $a$ encoding the choice of a $\spinc$-structure as input.
This series has two remarkable properties: it recovers the Witten-Reshetikhin-Turaev (WRT) invariants via certain appropriate limits \cite{murakami2024proof} and is known in some cases to be a quantum modular form \cite{lawrence1999modular, liles2023infinite}. 

The series $\widehat{Z}_a(q)$ is expected to be an instantiation of a 3D topological quantum field theory yet to be determined in general. This expectation has been supported in Gukov-Manolescu \cite{gukov2021two}, where 
an analogous series $\widehat{Z}_a(q,z)$ for knot complements has been introduced and shown to satisfy a gluing formula. Moreover, the series $\widehat{Z}_a(q)$ has been extended to include the datum of an arbitrary root lattice $Q$ in Park \cite{park2020higher}, 
a generalization that is motivated by the relationship between root systems and quantum groups and their role as inputs to the construction of WRT invariants.

In \cite{MT1}, we showed that the series $\widehat{Z}_a(q)$ decomposes as an average 
\[
\widehat{Z}_a(q) = \frac{1}{|\Xi|}\sum_{\xi\in \Xi} \mathsf{Y}_\tau\left(q\right) \qquad \mbox{with $\tau=(Q, a, \xi)$}
\]
with each series $\mathsf{Y}_\tau\left(q\right)$ invariant under the Neumann moves amongst plumbing trees. 
However, while $\widehat{Z}_a(q)$ is also invariant under the action of the Weyl group $W$ of $Q$, the summands $\mathsf{Y}_\tau\left(q\right)$ might not be so individually. This allows one to obtain distinct series for different $\spinc$-structures in the same orbit under the action of the Weyl group $W$.
Here $\Xi$ is an appropriate set of assignments $\xi$ of elements of the Weyl group $W$ of $Q$ to the vertices of the plumbing tree of the 3-manifold. 

For the series $\widehat{Z}_a(q)$ and $\widehat{Z}_a(q,z)$ to be a well-defined Laurent series, one requires that the framing matrix of the plumbing tree is definite --- or weakly definite, as defined in \cite{gukov2021two}. This assumption was also used in \cite{MT1} for the series $\mathsf{Y}_\tau\left(q\right)$. In the case of $Q=A_1$, this assumption was removed for the series  $\widehat{Z}_a(q)$ in Ri \cite{ri2023refined}, after introducing an additional variable $t$.

Here we show that the (weakly) definite assumption can be similarly removed for our refinements $\mathsf{Y}_\tau\left(q\right)$ and arbitrary root lattices by introducing a variable $t$. Thus we only require that the plumbing tree be reduced (as in \S\ref{subsec:reduced}) --- an assumption also needed in \cite{ri2023refined, MT1} and for all other results on $\widehat{Z}_a(q)$. Importantly, every plumbing tree can become reduced after a sequence of Neumann moves. 
We still assume throughout invertibility of the framing matrix  over $\mathbb{Q}$. \footnote{This hypothesis may be further removed by \cite[Rmk 4.4]{gukov2021two}.}
Thus in the closed case, we show:

\begin{thm}
\label{thm:qtseriesinvarianceintroreduced}
For a reduced plumbing tree $\Gamma$ and a tuple $\tau=(Q, a, \xi)$:
\begin{enumerate}[(i)]
\item the series $\mathsf{Y}_\tau\left(q, t \right)$ is invariant under the five Neumann moves between reduced plumbing trees, and
\item with respect to the action of the Weyl group $W$, one has
\begin{equation*}
\mathsf{Y}_\tau\left(q, t \right)=\mathsf{Y}_{w(\tau)}\left(q, t \right), \quad \mbox{for $w\in W$}
\end{equation*}
where $w(\tau):=(Q, w(a), w(\xi))$.
\end{enumerate}
\end{thm}

When the $(q,t)$-series can be evaluated at $t=1$, the resulting series $\mathsf{Y}_\tau\left(q, 1\right)$ recovers the $q$-series from \cite{MT1}, whose average over $\xi\in \Xi$ is the series $\widehat{Z}_a(q)$ that is invariant under the action of $W$. 
In the event that $\Gamma$ is negative-definite, the series $\mathsf{Y}_\tau(q,t)$ is invariant under the two Neumann moves between arbitrary (not necessarily reduced) negative-definite plumbing trees. 
Thus, for a negative-definite plumbing tree $\Gamma$ and $Q=A_1$, we recover the series $\Zdhat_a(q,t)$ from \cite{akhmechet2023lattice} as
\begin{equation}
\label{eq:Zdoublehat}
\Zdhat_a\left(q,t^2\right) = \frac{1}{2^{|V(\Gamma)|}}\sum_{\xi\in W^{V(\Gamma)}}\mathsf{Y}_\tau\left(q, t^\xi\right) \qquad\mbox{with 
$\tau=(A_1, a, \xi)$}.
\end{equation}
See notation in \eqref{eq:txi}.

Similarly, for a knot complement obtained by a plumbing tree $\Gamma$ with a distinguished vertex $v_0$ (see plumbing pairs in \S\ref{sec:neumann-knots}), we have:

\begin{thm}
\label{thm:qtzseriesinvarianceintroreduced}
For a reduced pair $(\Gamma, v_0)$ and a tuple $\tau=(Q, a, \xi)$:
\begin{enumerate}[(i)]
\item the series $\mathsf{Y}_\tau\left(q, t, z\right)$ is invariant under the five Neumann moves between reduced plumbing pairs, and
\item with respect to the action of the Weyl group $W$, one has
\begin{equation*}
\mathsf{Y}_\tau\left(q, t, z\right)=\mathsf{Y}_{w(\tau)}\left(q, t, z\right), \quad \mbox{for $w\in W$}
\end{equation*}
where $w(\tau):=(Q, w(a), w(\xi))$.
\end{enumerate}
\end{thm}

When the $(q,t,z)$-series can be evaluated at $t=1$, we recover the $(q,z)$-series for $Q=A_1$ from \cite{gukov2021two} as
\[
\widehat{Z}_a(q,z) = \frac{1}{2^{|V(\Gamma)|}}\sum_{\xi\in W^{V(\Gamma)}}\mathsf{Y}_\tau\left(q, 1, z\right) \qquad\mbox{with 
$\tau=(A_1, a, \xi)$}.
\]

Next, we show that the above series for closed 3-manifolds and knot complements verify a gluing formula.
Assume $M$ is obtained by gluing a pair of plumbed knot complements  
$\left(M^\pm, \partial M^\pm\right):=M\left(\Gamma^\pm, v^\pm_0\right)$ along their boundaries. Given a $\spinc$-structure $a$, select
 relative $\mathrm{Spin}^c$-structures $a^\pm$ on $\left(M^\pm, \partial M^\pm\right)$ which glue to $a$; see \eqref{eq:spinciso}. 
Starting from $\xi$, define $\xi^\pm$ to be the restriction of $\xi$ to $\Gamma^\pm$.

\begin{thm}[A gluing formula]
\label{thm:gluingintroreduced}
One has
\begin{align*}
\mathsf{Y}_\tau\left(M; q, t \right) = 
(-1)^\triangle q^\square  
\sum_{\gamma \in Q}
 \left[\mathsf{Y}^+_{\gamma}(z)  \, \mathsf{Y}^-_{\gamma}(z)  \right]_0
\end{align*}
where $\triangle$ and $\square$ are given in \eqref{eq:trianglesquare}, and
\[
	\mathsf{Y}^{\pm}_{\gamma}(z):= \mathsf{Y}_{\tau^\pm}\left(M^\pm; q, t, z\right) \qquad \mbox{with $\tau^\pm = \tau^\pm(\gamma) := \left(Q, b^\pm, \xi^\pm \right)$}
\]
for $\gamma\in Q$, and $b^\pm$ depending on $a^\pm$ and $\gamma$ as in \eqref{eq:bpm}.
\end{thm}

The operator $[\,\,]_0$ appearing in the statement assigns to a series in $z$ the constant term in $z$.

Finally, we verify how the $(q,t)$-series varies under  the Neumann splitting move in Figure \ref{fig:splitting}. 
Namely, for certain plumbing trees $\Gamma_1$ and $\Gamma_2$ obtained by splitting a plumbing tree $\Gamma_\circ$ and tuples $\tau_1$ and $\tau_2$ obtained by splitting a tuple $\tau_\circ$,
we show that the $(q,t)$-series for $\Gamma_\circ$ decomposes as a sum of products of certain restrictions $\mathsf{Y}^w_{\tau_i}\left(\Gamma_i; q,t\right)$ of the $(q,t)$-series for $\Gamma_1$ and $\Gamma_2$ times an additional $(q,t)$-series:

\begin{thm}[A splitting formula]
\label{thm:Du2intro}
One has
\[
\mathsf{Y}_{\tau_\circ}\left(\Gamma_\circ; q,t\right) 
= \sum_{w\in W}
\mathsf{Y}^w_{\tau_1}\left(\Gamma_1; q,t\right) \mathsf{Y}^w_{\tau_2}\left(\Gamma_2; q,t\right) \mathsf{R}_{w, \tau_\circ}(q,t)
\]
where $\mathsf{R}_{w, \tau_\circ}(q,t)$ is an explicit $(q,t)$-series given in Theorem \ref{thm:Du2}.
\end{thm}

Evidently, the $(q,t)$-series is not invariant under the splitting move. Thus we pose the question:

\begin{q}
How can one modify the $(q,t)$-series so that it becomes invariant under all Neumann moves between forests?
\end{q}

\subsection*{Structure of the paper}
After reviewing the required background in \S\ref{sec:background}, 
we define Weyl assignments $\xi$ in \S\ref{sec:Weylassignments}.
The $(q,t)$-series for closed 3-manifolds is defined in \S\ref{sec:qtseries}. Theorem \ref{thm:qtseriesinvarianceintroreduced} follows from Theorems \ref{thm:qtseries} and \ref{thm:Winvqtseries}.
The $(q,t,z)$-series for knot complements is defined in \S\ref{sec:knotinvt}. Theorem \ref{thm:qtzseriesinvarianceintroreduced} is proven there. Theorem \ref{thm:gluingintroreduced} follows from Theorem \ref{thm:gluing} and Theorem \ref{thm:Du2intro} from Theorem \ref{thm:Du2}. The case of lens spaces and Brieskorn spheres is explicitly discussed in \S\ref{sec:ex}.

%%%%%%%%%%%%%%%%%%%%%%%%%%%%
%%%%%%%%%%%%%%%%%%%%%%%%%%%%
%%%%%%%%%%%%%%%%%%%%%%%%%%%%

\section{Background}
\label{sec:background}

In this section, we review the required background on plumbed $3$-manifolds,  
root lattices, and $\mathrm{Spin}^c$-structures.

\subsection{Plumbed 3-manifolds}
\label{sec:plumbing}
The input of our invariant will be a \emph{plumbing tree} $\Gamma$ consisting of a vertex set $V(\Gamma)$, an edge set $E(\Gamma)$, and integer-valued vertex weights $m_v$ for $v\in V(\Gamma)$. 
 Following the plumbing construction as in Neumann \cite{neumann1981calculus} (see also  \cite[\S3.3]{MR4510934}),  a plumbing graph $\Gamma$ gives rise to a closed oriented $3$-manifold $M(\Gamma)$ as follows. One assigns to each vertex  an oriented disk bundle over the sphere with Euler number $m_v$, with two such bundles  plumbed together when the corresponding vertices are connected by an edge in $\Gamma$. This construction yields a $4$-manifold $X=X(\Gamma)$, the boundary of which is the \emph{plumbed 3-manifold} ${M}={M}(\Gamma)$. 

Alternatively, the plumbing construction may be realized by performing Dehn surgery along a framed link. Specifically, the \emph{framed link} $L(\Gamma)$ corresponding to a plumbing tree $\Gamma$ consists of an unknotted component with framing $m_v$ for each vertex $v$ of $\Gamma$, with two unknotted components chained together whenever the corresponding vertices in $\Gamma$ are connected by an edge.

Neumann showed that 
%any orientation preserving homeomorphism between two arbitrary plumbed 3-manifolds decomposes as a sequence of combinatorial moves between their plumbing graphs
if two plumbing graphs represent the same 3-manifold up to orientation-preserving diffeomorphism, then they are related by a finite sequence of combinatorial moves \cite{neumann1981calculus}.
The only such moves between two plumbing \textit{trees} are the five moves given in Figure \ref{fig:Neumann} and their inverses. 

\begin{figure}[t]
\hspace{-0.8cm}
\begin{subfigure}[t]{0.26\textwidth}
\[
\begin{array}{c}
\begin{tikzpicture}[baseline={([yshift=0ex]current bounding box.center)}]
      \path(0,0) ellipse (2 and 2);
      \tikzstyle{level 1}=[counterclockwise from=150,level distance=15mm,sibling angle=30]
      \node [draw,circle, fill, inner sep=1.5, label={[label distance=10]90:$m_1\pm1$}] (A1) at (180:4) {}
            child {node [label=: {}]{}}
	    child {node [label=: {}]{}}
    	    child {node [label=: {}]{}};
      \tikzstyle{level 1}=[counterclockwise from=30,level distance=15mm,sibling angle=-60]
      \node [draw,circle,fill, inner sep=1.5, label={[label distance=10]90:$m_2\pm1$}] (A2) at (0:4) {}
	   child {node [label=: {}]{}}
	   child {node [label=: {}]{}};
      \tikzstyle{level 1}=[counterclockwise from=150,level distance=15mm,sibling angle=30]
      \node [draw,circle, fill, inner sep=1.5, label={[label distance=10]90:$\pm1$}] (A0) at (0:0) {};
      \path (A1) edge []  node[midway, label={[label distance=10]-90:}]{} (A0);
       \path (A2) edge []  node[midway, label={[label distance=10]-90:}]{} (A0);
    \end{tikzpicture}
\\%[1cm]
\rotsimeq
    \\
\begin{tikzpicture}[baseline={([yshift=0ex]current bounding box.center)}]
      \path(0,0) ellipse (2 and 2.5);
      \tikzstyle{level 1}=[counterclockwise from=150,level distance=15mm,sibling angle=30]
      \node [draw,circle, fill, inner sep=1.5, label={[label distance=8]90:$m_1$}] (A1) at (180:2) {}
            child {node [label=: {}]{}}
	    child {node [label=: {}]{}}
    	    child {node [label=: {}]{}};
      \tikzstyle{level 1}=[counterclockwise from=30,level distance=15mm,sibling angle=-60]
      \node [draw,circle,fill, inner sep=1.5, label={[label distance=8]90:$m_2$}] (A2) at (0:2) {}
	   child {node [label=: {}]{}}
	   child {node [label=: {}]{}};
      \path (A1) edge []  node[midway, label={[label distance=10]-90:}]{} (A2);
    \end{tikzpicture}
\end{array}
\]
\caption*{(A$\pm$)}
\end{subfigure}
\qquad\quad
\begin{subfigure}[t]{0.26\textwidth}
\[
\begin{array}{c}
\begin{tikzpicture}[baseline={([yshift=0ex]current bounding box.center)}]
      \path(0,0) ellipse (2 and 2);
      \tikzstyle{level 1}=[counterclockwise from=150,level distance=15mm,sibling angle=30]
      \node [draw,circle, fill, inner sep=1.5, label={[label distance=10]90:$m_1\pm1$}] (A1) at (180:4) {}
            child {node [label=: {}]{}}
	    child {node [label=: {}]{}}
    	    child {node [label=: {}]{}};
      \tikzstyle{level 1}=[counterclockwise from=150,level distance=15mm,sibling angle=30]
      \node [draw,circle, fill, inner sep=1.5, label={[label distance=10]90:$\pm1$}] (A0) at (0:0) {};
      \path (A1) edge []  node[midway, label={[label distance=10]-90:}]{} (A0);
    \end{tikzpicture}
\\%[1cm]
\rotsimeq
    \\
\begin{tikzpicture}[baseline={([yshift=0ex]current bounding box.center)}]
      \path(0,0) ellipse (2 and 2.5);
      \tikzstyle{level 1}=[counterclockwise from=150,level distance=15mm,sibling angle=30]
      \node [draw,circle, fill, inner sep=1.5, label={[label distance=8]90:$m_1$}] (A1) at (0:0) {}
            child {node [label=: {}]{}}
	    child {node [label=: {}]{}}
    	    child {node [label=: {}]{}};
    \end{tikzpicture}
\end{array}
\]
\caption*{(B$\pm$)}
\end{subfigure}
%\quad
\begin{subfigure}[t]{0.26\textwidth}
\[
\begin{array}{c}
\begin{tikzpicture}[baseline={([yshift=0ex]current bounding box.center)}]
      \path(0,0) ellipse (2 and 2);
      \tikzstyle{level 1}=[counterclockwise from=150,level distance=15mm,sibling angle=30]
      \node [draw,circle, fill, inner sep=1.5, label={[label distance=10]90:$m_1$}] (A1) at (180:4) {}
            child {node [label=: {}]{}}
	    child {node [label=: {}]{}}
    	    child {node [label=: {}]{}};
      \tikzstyle{level 1}=[counterclockwise from=30,level distance=15mm,sibling angle=-60]
      \node [draw,circle,fill, inner sep=1.5, label={[label distance=10]90:$m_2$}] (A2) at (0:4) {}
	   child {node [label={[label distance=5]0:}]{}}
	   child {node [label={[label distance=5]0:}]{}};
      \tikzstyle{level 1}=[counterclockwise from=150,level distance=15mm,sibling angle=30]
      \node [draw,circle, fill, inner sep=1.5, label={[label distance=10]90:$0$}] (A0) at (0:0) {};
      \path (A1) edge []  node[midway, label={[label distance=10]-90:}]{} (A0);
       \path (A2) edge []  node[midway, label={[label distance=10]-90:}]{} (A0);
    \end{tikzpicture}
\\%[1cm]
\rotsimeq
    \\
\begin{tikzpicture}[baseline={([yshift=0ex]current bounding box.center)}]
      \path(0,0) ellipse (2 and 2.5);
      \tikzstyle{level 1}=[counterclockwise from=150,level distance=15mm,sibling angle=30]
      \node [draw,circle, fill, inner sep=1.5, label={[label distance=10]90:$m_1+m_2$}] (A1) at (0:0) {}
            child {node [label=: {}]{}}
	    child {node [label=: {}]{}}
    	    child {node [label=: {}]{}};
      \tikzstyle{level 1}=[counterclockwise from=30,level distance=15mm,sibling angle=-60]
      \node [draw,circle,fill, inner sep=1.5, label={[label distance=8]90:}] (A2) at (0:0) {}
	   child {node [label={[label distance=10]0:}]{}}
	   child {node [label={[label distance=10]0:}]{}};
    \end{tikzpicture}
\end{array}
\]
\caption*{(C)}
\end{subfigure}

\caption{The five Neumann moves on plumbing trees.}
\label{fig:Neumann}
\end{figure}
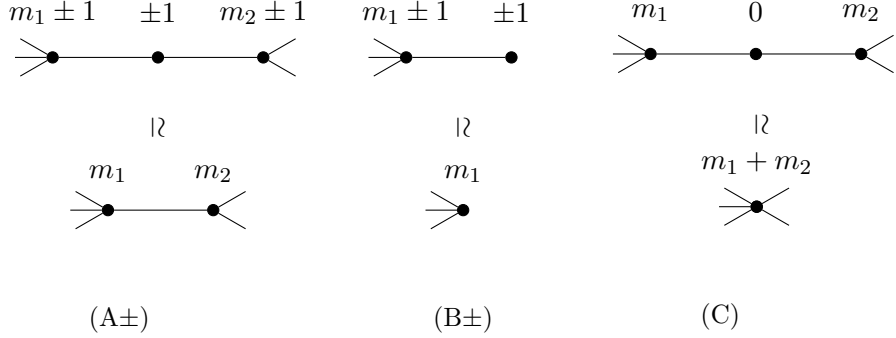

\subsection{Reduced plumbing trees}
\label{subsec:reduced}
We use reduced plumbing trees as in \cite{ri2023refined}. These are defined as follows.
For a plumbing tree $\Gamma$, a subtree of $\Gamma$ is said to be 
\textit{(Neumann) contractible} if it can be contracted down to a single vertex by a sequence of the Neumann moves from Figure \ref{fig:Neumann}. Contractibility can be characterized using the arithmetic of continued fractions \cite[Lemma 1.6]{MT1}.
A vertex $v$ of $\Gamma$ is said to be \textit{reducible} if $v$ has degree at least $3$ but, after contracting all contractible subtrees incident to $v$, the degree of $v$ drops down to $1$ or $2$.
Finally,  $\Gamma$ is said to be \textit{reduced} if $\Gamma$ has no reducible vertices. 
\textit{Any plumbing tree can be reduced} via a sequence of the Neumann moves from Figure \ref{fig:Neumann}.

A result from \cite{ri2023refined} shows that two reduced plumbing trees are related by a sequence of the Neumann moves from Figure \ref{fig:Neumann} if and only if they are related by a sequence of such moves between reduced plumbing trees (see also \cite[Prop.~1.8]{MT1} for an extended argument).
For an example of such moves, consider the case when $\Gamma$ consists of a single vertex. Then $\Gamma$ is reduced, and any of the moves (B$\pm$) or (C) yields a plumbing tree which is also reduced.

We will thus define our series starting from reduced plumbing trees and show that it is invariant under the Neumann moves between reduced plumbing trees.

\subsection{Homology of a plumbed 3-manifold}
\label{sec:hom}
For a plumbing tree $\Gamma$, select an order of its vertices $v_1, \dots, v_s$, with $s=|V(\Gamma)|$. 
Then $\Gamma$ determines a symmetric $s\times s$ matrix $B$, called the \textit{framing matrix}:
\[
B:= (B_{ij})_{i,j=1}^s \quad \mbox{ with }\quad
B_{ij}:=
\left\{
\begin{array}{ll}
m_i & \mbox{if $i=j$,}\\[5pt]
1 & \mbox{if $i\neq j$ and $(i, j)\in E(\Gamma)$,}\\[5pt]
0 & \mbox{otherwise}
\end{array}
\right.
\]
where $m_i$ is the weight of $v_i$, and $(i, j)$ denotes an edge between $v_i$ and $v_j$.
We will denote the signature of $B$ by $\sigma = \sigma(B)$ and the number of its positive eigenvalues by $\pi= \pi(B)$. 
 
The  matrix $B$ is the natural intersection pairing on $L:= H_2\left( X; \mathbb{Z} \right) \cong \mathbb{Z}^s$. Moreover,
 $B$ realizes the natural inclusion $L\hookrightarrow L'$, where $L' := H^2\left( X; \mathbb{Z} \right)\cong H_2\left( X, {M}; \mathbb{Z} \right)\cong \mathbb{Z}^s$ is the dual lattice. A standard homological argument shows
\begin{equation}
\label{eq:H1}
H_1\left( M; \mathbb{Z} \right) \cong L'/L \cong \mathbb{Z}^s/B\mathbb{Z}^s.
\end{equation}
We will assume throughout that $\det (B)\neq 0$. In particular, $B$ has maximal rank, hence $M$ is a rational homology sphere, i.e., $H_1\left( M; \mathbb{Q} \right)=0$.

\subsection{Root lattices and $\spinc$-structures}
\label{sec:knots}
For a general treatment of root lattices, we refer to  \cite{MR1890629, MR0323842}.
Let $Q$ be a root lattice of rank $r$ with root system $\Delta$.
Let $\Delta^+$ be a set of \textit{positive roots} of $\Delta$. The \textit{Weyl vector} $\rho\in \frac{1}{2}Q$ is defined to be half the sum of the positive roots. The \textit{Weyl group} $W$ acting on $Q$ is the group generated by reflections through the hyperplanes orthogonal to the roots. The \textit{length} $\ell(w)$ of an element $w\in W$ is its word length expressed as a product of reflections and is equal to the number of positive roots transformed by $w$ into negative roots.

For a plumbing tree $\Gamma$ such that $\det(B)\neq 0$, the induced bilinear pairing on the lattice $L'\otimes_{\mathbb{Z}}Q\cong Q^s$ is 
\begin{equation*}
\langle, \rangle \colon Q^s \times Q^s \rightarrow \mathbb{Q}, \qquad 
\langle a, b \rangle = \sum_{i,j=1}^s \left( B^{-1}\right)_{ij} \langle a_i, b_j\rangle
\end{equation*}
where $\langle a_i, b_j\rangle$ is the pairing in $Q$.
The space of $\spinc$-structures on $M=M(\Gamma)$ with coefficients in $Q$ is
\begin{equation}
\label{spinc-isomorphism}
\spinc_Q\left(M \right) := \frac{\delta +2Q^{s}}{2Q\langle B_1, \dots, B_{s}\rangle},
\end{equation}
where $B_i$ is the $i$-th column of $B$, and 
\[
	\delta:=\left(2-\deg(v_1), \cdots, 2-\deg(v_s) \right)\otimes 2\rho \in \mathbb{Z}^s  \otimes_{\mathbb{Z}} Q\cong Q^s.
\]
One has an affine isomorphism $\spinc_Q\left(M \right)  \cong H_1(M; Q)$ (this is clear from \eqref{eq:H1}). 
Thus, the space $H_1(M; Q)$ naturally acts on $\spinc_Q\left(M \right)$ via 
\[
[x]\cdot[a]=[a+2x] \qquad \mbox{for $x\in Q^s$ and $a\in\delta +2Q^{s}$.}
\]
Also, the Weyl group $W$ acts component-wise on $Q^s$, and this induces an action of $W$ on $\spinc_Q\left(M \right)$.

While \eqref{spinc-isomorphism} uses the choice of a plumbing tree $\Gamma$ for $M$, the resulting set and the  action of $W$ on it are invariant under the five Neumann moves in Figure \ref{fig:Neumann} (see \cite[Prop.~1.2]{MT1} for an explicit proof).

\subsection{Plumbed knot complements}
\label{sec:neumann-knots}
We will also be interested in plumbed 3-manifolds with boundary homeomorphic to a torus, i.e., the complement of a knot in a plumbed 3-manifold. We refer to \cite[\S5]{gukov2021two} for a general reference on the topics reviewed here and in the next two subsections. 

The plumbing presentation for such a 3-manifold $(M, \partial M)$ consists of a pair $(\Gamma, v_0)$ where $\Gamma$ is a plumbing tree and $v_0$ is a distinguished vertex of $
\Gamma$. We refer to such a $(\Gamma, v_0)$ as a \textit{plumbing pair} or simply a pair. The component corresponding to $v_0$ in the framed link $L(\Gamma\setminus v_0)$ represents a knot $K$ in $M(\Gamma\setminus v_0)$. Then $M=M(\Gamma, v_0)$ is defined as the complement of a tubular neighborhood of $K$ in $M(\Gamma\setminus v_0)$. It follows that $M$ is a 3-manifold with a torus boundary $\partial M$.

 Moreover, the plumbing presentation specifies a parametrization of $\partial M$. Indeed, the  presentation specifies the meridian $\mu$ of the knot and a longitude $\lambda$ given by the framing of  $K$ determined by the weight of $v _0$ in $\Gamma$ (this is the \emph{graph longitude} from \cite{gukov2021two}). One orients $\lambda$ counterclockwise, while the orientation of $\mu$ is uniquely determined from the boundary orientation of $\partial M$ induced from the orientation of $M$.

As with closed plumbed 3-manifolds, if two pairs $(\Gamma, v_0)$ and $(\Gamma', v'_0)$ represent the same 3-manifold with 
parametrized boundary $(M, \partial M)$ up to orientation-pre\-serving diffeomorphism, then $(\Gamma, v_0)$ and $(\Gamma', v'_0)$ are related by a finite sequence of combinatorial moves from \cite{neumann1981calculus}.
The only such moves between two plumbing \textit{trees} with a distinguished vertex are the five moves given in Figure \ref{fig:Neumann} and their inverses. The distinguished vertex can be involved in one of such moves, as long as it is not one of the vertices weighted by $\pm 1$ or $0$ in the top plumbing trees in Figure \ref{fig:Neumann}. We refer to such moves as \textit{Neumann moves on (plumbing) pairs}.

\subsection{Reduced plumbing pairs}
\label{sec:redplpairs}
Modifying \S\ref{subsec:reduced}, 
given a plumbing pair $(\Gamma, v_0)$, a subtree of $\Gamma$ is said to be 
\textit{relatively contractible} if it can be contracted down to a single vertex by a sequence of the Neumann moves \textit{on pairs}. 
In particular, $v_0$ will not be removed by such a contraction.
A vertex $v$ of $\Gamma$ is said to be \textit{relatively reducible} if $v$ has degree at least $3$ but, after contracting all relatively contractible subtrees incident to $v$, the degree of $v$ drops down to $1$ or $2$.
Finally, $(\Gamma, v_0)$ is said to be \textit{reduced} if $\Gamma$ has no relatively reducible vertices.

Any pair $(\Gamma, v_0)$ can be reduced via a sequence of Neumann moves on pairs.
Furthermore, two reduced plumbing pairs are related by a sequence of the Neumann moves on pairs if and only if they are related by a sequence of such moves between reduced plumbing pairs. Indeed, this follows by replacing contractibility with relative contractibility in the proof of \cite[Prop.~1.8]{MT1}).

\subsection{Relative $\spinc$-structures}
\label{sec:relspinc}
For a pair $(\Gamma, v_0)$, select an order $v_0, \dots, v_s$ of the vertices of $\Gamma$, with $s+1=|V(\Gamma)|$, and let $B$ be the framing matrix of $\Gamma$. The space of relative $\spinc$-structures for $(M, \partial M)$ is
\begin{equation}
\label{spinc-isomorphism-relative}
\spinc_Q\left(M, \partial M \right) := \frac{\widehat{\delta} +2Q^{s+1}}{2Q\langle B_1, \dots, B_{s}\rangle},
\end{equation}
where the zero-th column of $B$ corresponding to $v_0$ is omitted in the denominator and $\widehat{\delta}:=\delta-(2\rho,0,\dots,0)$, i.e.,
\[
\widehat{\delta}:=\left(1-\deg(v_0), 2-\deg(v_1), \cdots, 2-\deg(v_s) \right)\otimes 2\rho \in \mathbb{Z}^{s+1}  \otimes_{\mathbb{Z}} Q\cong Q^{s+1}.
\]
One has an affine isomorphism $\spinc_Q\left(M, \partial M \right)  \cong H_1(M, \partial M; Q)$. 
Also, the component-wise action of the Weyl group $W$ on $Q^s$ induces an action of $W$ on $\spinc_Q\left(M, \partial M \right)$.

In \eqref{spinc-isomorphism-relative}, the shift by $\widehat{\delta}$ can be replaced with a shift by ${\delta}$ as in \eqref{spinc-isomorphism} --- this is the convention used in \cite{gukov2021two}.
However, when doing so, the action of $W$ on the resulting identification of $\spinc_Q\left(M, \partial M \right)$ is not given by $[a]\mapsto [w(a)]$ for $w\in W$ and $a\in \delta +2Q^{s+1}$, as observed for $Q=A_1$ in \cite[Rmk 2.7]{akhmechet2025knot}.

\subsection{Gluing knot complements}
\label{sec:gluing}
Consider a pair of plumbed knot complements 
\[
\left(M^+, \partial M^+\right):=M\left(\Gamma^+, v^+_0\right) \quad \mbox{and} \quad \left(M^-,  \partial M^-\right):=M\left(\Gamma^-, v_0^-\right).
\] 
As in \S \ref{sec:neumann-knots}, the plumbing presentations $\left(\Gamma^\pm, v^\pm_0\right)$ specify parametrizations of $\partial M^{\pm}$ with oriented meridians $\mu^\pm$ and longitudes $\lambda^\pm$. 
Let 
\[
	h\colon\partial M^+\rightarrow \partial M^-
\]
be the orientation-reversing homeomorphism induced by  $\lambda^+ \mapsto\lambda^-$ and $\mu^+\mapsto -\mu^-$.
Gluing $M^+$ and $M^-$ along their boundaries via $h$ yields a closed oriented 3-manifold
\begin{equation}
\label{eq:Mgluing}
M:=\left(M^+, \partial M^+\right)\cup_h \left(M^-, \partial M^-\right).
\end{equation}
This is a plumbed 3-manifold $M\cong M(\Gamma)$, where $\Gamma$ is the plumbing tree obtained by identifying the vertex $v^+_0$ of $\Gamma^+$ with the vertex $v^-_0$ of $\Gamma^-$. The weight of the resulting vertex $v_0$ is defined to be equal to the sum of the weights of the vertex $v^+_0$ in $\Gamma^+$ and the vertex $v^-_0$ in $\Gamma^-$ \cite[\S 5.1]{gukov2021two}.

To express the framing matrix of $\Gamma$ in terms of the framing matrices $B^\pm$ of $\Gamma^\pm$, 
select an order $v_1, \dots, v_m$ of the vertices of $\Gamma^+$, with $m=|V(\Gamma^+)|$, such that $v_m$ is the distinguished vertex, and an order $v_1, \dots, v_n$ of the vertices of $\Gamma^-$, with $n=|V(\Gamma^-)|$, such that $v_1$ is the distinguished vertex.
Consider the operation $\ast \colon Q^m\times Q^n\rightarrow Q^{s}$, where $s=m+n-1=|V(\Gamma)|$, defined as 
\begin{equation}
\label{ast}
a^+ \ast a^-:=\left(a_1^+, \dots a_{m-1}^+, a_{m}^+ + a_1^-, a_2^-, \dots, a_{n}^- \right).
\end{equation}
The framing matrix of $\Gamma$ is then given by the matrix 
\begin{equation}
\label{eq:Bast}
B=B^+\ast B^-
\end{equation}
defined by
\[
B_i := \left\{
\begin{array}{ll}
B^+_i \ast \bm{0}& \mbox{for $i\in \{1,\dots, m-1\}$,}\\[5pt]
B^+_m \ast  B^- _1 & \mbox{for $i=m$,}\\[5pt]
\bm{0} \ast B^-_{1+i-m} & \mbox{for $i\in \{m+1,\dots, s\}$.}
\end{array}
\right.
\]

\subsection{$\spinc$-structures under gluing}
\label{subsec:spinc-gluing}
For a root lattice $Q$, consider identifications
\begin{align}
\label{eq:idspincgluing}
\begin{split}
\spinc_Q\left(M^+, \partial M^+\right) &= \frac{\widehat{\delta}^+ +2Q^{m}}{2Q\langle B^+_1, \dots, B^+_{m-1}\rangle},\\
\spinc_Q\left(M^-, \partial M^-\right) &= \frac{\widehat{\delta}^- +2Q^{n}}{2Q\langle B^-_2, \dots, B^-_{n}\rangle},\\
\spinc_Q\left(M\right) &= \frac{\delta +2Q^{s}}{2Q\langle B_1, \dots, B_{s}\rangle}
\end{split}
\end{align}
as in \eqref{spinc-isomorphism} and \eqref{spinc-isomorphism-relative}, where $\widehat{\delta}^\pm$ and $\delta$ are defined accordingly.
The Mayer-Vietoris sequence for  the gluing \eqref{eq:Mgluing} induces  a surjective map
\begin{equation}
\label{eq:MaVi}
\spinc_Q\left(M^+, \partial M^+\right) \oplus \spinc_Q\left(M^-, \partial M^-\right) \rightarrow \spinc_Q(M)
\end{equation}
given by $[a^+] \oplus[ a^-] \mapsto [a^+ \ast a^-]$,
where the operation $\ast$ is as in \eqref{ast}. 
This map is independent of the choice of representatives (the case $Q=A_1$ is in \cite[\S 5.4]{gukov2021two}). 
Moreover, the Mayer-Vietoris sequence induces an action of $H_1(\partial M^+;Q)\cong Q\langle\lambda, \mu\rangle$ on the source of the map \eqref{eq:MaVi} given by
\begin{align}
\label{eq:lamuaction}
\begin{split}
\gamma\lambda\colon \left[a^+\right] \oplus \left[ a^-\right] &\mapsto \left[a^+ + 2\gamma B^+_m\right] \oplus \left[ a^- + 2\gamma B^-_1\right],\\
\gamma\mu\colon \left[a^+\right] \oplus \left[ a^-\right] &\mapsto \left[a^+ +(0,\dots,0,2\gamma)\right] \oplus \left[ a^- -(2\gamma,0,\dots,0)\right]
\end{split}
\end{align}
for $\gamma\in Q$. Factoring by this action, the map \eqref{eq:MaVi} induces an isomorphism
\begin{equation}
\label{eq:spinciso}
\frac{\spinc_Q\left(M^+, \partial M^+\right) \oplus \spinc_Q\left(M^-, \partial M^-\right)}{H_1(\partial M^+;Q)}
\xrightarrow{\cong} \spinc_Q(M).
\end{equation}

%%%%%%%%%%%%%%%%%%%%%%%%%%%%
%%%%%%%%%%%%%%%%%%%%%%%%%%%%
%%%%%%%%%%%%%%%%%%%%%%%%%%%%

\section{Weyl assignments}
\label{sec:Weylassignments}

Here we define Weyl assignments on reduced plumbing trees. These are used in the definition of the $(q,t)$-series  as an input for both the coefficients  of the series and the exponent of the variable $t$.

Let $Q$ be a root lattice with Weyl group $W$.
For a reduced plumbing tree $\Gamma$ with framing matrix $B$, define a \textit{Weyl assignment} to be a map
\[
\xi\colon V(\Gamma) \rightarrow W, \qquad v\mapsto \xi_v
\]
such that 
\[
\xi_v = 1_W \qquad \mbox{if $\deg v = 2$},
\]
where $1_W$ is the identity element in $W$,
and such that the values on vertices across what we call maximal contractible degree-$2$ paths are coordinated by the following condition \eqref{eq:forcingbridgescondition}.

First some notation. A path in $\Gamma$ is said to have \textit{degree $2$} if all its vertices have degree $2$ in $\Gamma$, except the two terminal vertices which can have arbitrary degree in $\Gamma$. As in \S\ref{subsec:reduced}, a path in $\Gamma$ is \textit{contractible} if it can be contracted down to a single vertex by a sequence of the Neumann moves from Figure~\ref{fig:Neumann}. This can be characterized using continued fractions \cite[Lemmata 1.6 and 3.1]{MT1}.
A contractible degree-$2$ path is \textit{maximal} if it is not a proper subpath of a contractible degree-$2$ path.

For a maximal contractible degree-$2$ path $\Gamma_{v,v'}$ with terminal vertices $v$ and $v'$ in $\Gamma$ 
such that $\deg v\neq 2$ or $\deg v'\neq 2$,
the values of $\xi$ at $v$ and $v'$ are coordinated by the following condition:
\begin{equation}
\label{eq:forcingbridgescondition}
\xi_v = \iota^{\Delta\pi(v,v')}  \xi_{v'},
\end{equation}
where $\iota$ is the element of $W$ defined by
\begin{equation}
\label{eq:iota}
\iota(\alpha)=-\alpha \qquad\mbox{for all }\alpha\in Q,
\end{equation}
and $\Delta\pi(v,v')$ is the difference in numbers of positive eigenvalues
\begin{equation}
\label{eq:deltapibridge}
\Delta\pi(v,v') := \pi(B)- \pi(\overline{B}),
\end{equation}
with $\overline{B}$ equal to the framing matrix of the plumbing tree obtained from $\Gamma$ by contracting $\Gamma_{v,v'}$.

Note that for a degree-$2$ path $\Gamma_{v,v'}$, the map $\xi$ assigns $1_W$ to all vertices of $\Gamma_{v,v'}$ different than $v$ and $v'$. Moreover, if $v$ has degree $2$ in $\Gamma$, then necessarily $\xi_v=1_W$, and thus $\xi_{v'}\in \{1_W, \iota\}$ by \eqref{eq:forcingbridgescondition}, and similarly if $v'$ has degree $2$ in $\Gamma$, then $\xi_{v}\in \{1_W, \iota\}$.

Let 
\[
\Xi:= \{\mbox{Weyl assignments $\xi$ on $\Gamma$}\}.
\]
One has $|\Xi| = |W|^n$ where 
\[
n:= |\{v\in V(\Gamma) : \deg v\neq 2\}| - |\{\mbox{max.~contractible deg-2 paths}\}|.
\]

The assumption that our plumbing trees are \emph{reduced} is crucial when comparing the sets of Weyl assignments between two  plumbing trees:

\begin{lemma}
\label{lemma:isoxi}
For two reduced plumbing trees $\Gamma$ and $\Gamma'$ related by a finite sequence of the Neumann moves from Figure \ref{fig:Neumann}, the sets of Weyl assignments on $\Gamma$ and $\Gamma'$ are isomorphic.
\end{lemma}

We will prove this statement and apply it in the proof of the next Theorem \ref{thm:qtseries}, where we exhibit an explicit isomorphism $S$ between the Weyl assignments on two reduced plumbing trees related by a Neumann move from Figure \ref{fig:Neumann}.
In particular, a Neumann move between two reduced plumbing trees does not create a reducible vertex which could increase the size of the set of Weyl assignments. 

\begin{remark}
\label{rmk:Wass}
The present definition of Weyl assignments is a refinement of the definition appearing in \cite[\S 3.2]{MT1} in the sense that 
the values at degree-$1$ and degree-$0$ vertices are possibly arbitrary here, subject to \eqref{eq:forcingbridgescondition}. 

Also, while the idea of using Weyl assignments here originates from the study of the case $Q=A_1$ in \cite{ri2023refined}, the present Weyl assignments for $Q=A_1$ differ from the analogous combinatorial feature used in \cite{ri2023refined}, where degree-1 vertices were assigned possibly a value $0$ in addition to signs $\pm 1$ corresponding to elements of the Weyl group $W\cong\{\pm1\}$ for $Q=A_1$.
\end{remark}

%%%%%%%%%%%%%%%%%%%%%%%%%%%%
%%%%%%%%%%%%%%%%%%%%%%%%%%%%
%%%%%%%%%%%%%%%%%%%%%%%%%%%%

\section{An invariant two-variable series}
\label{sec:qtseries}

Here we define a two-variable series and prove Theorem \ref{thm:qtseriesinvarianceintroreduced}.

\subsection{The Kostant collection}
\label{sec:K}
Consider the formal series
\begin{equation}
\label{eq:Wz}
K(z):= \prod_{\alpha\in \Delta^+}\left(\sum_{i\geq 0} z^{-(2i+1)\alpha} \right).
\end{equation}
Here $z^\alpha$ for a root $\alpha$ is a multi-index monomial defined as
\begin{equation}
\label{eq:zpowers}
z^\alpha :=\prod_{i=1}^r z_i^{\langle \alpha^\vee, \lambda_i\rangle}
\end{equation}
with $\alpha^\vee:=\frac{2}{\langle \alpha, \alpha\rangle}\alpha$ being the coroot of $\alpha$ and
$\lambda_1,\dots,\lambda_r$ being the fundamental weights. Hence $K(z)\in \mathbb{Z}\left\llbracket z_1^{- 1}, \dots, z_r^{- 1}\right\rrbracket$, the ring of formal series in variables $z_1^{- 1}, \dots, z_r^{- 1}$.

Expanding, one has
\[
K(z) = \sum_{\alpha\in Q} k(\alpha)\, z^{-2\rho-2\alpha}
\]
where $k(\alpha)$ is the \textit{Kostant partition function} defined as
\begin{equation}
\label{eq:Kostant}
k(\alpha):= 
\begin{array}{l}
\mbox{number of ways to represent $\alpha$} \\ 
\mbox{as a sum of positive roots.}
\end{array}
\end{equation}
A key property of the series $K(z)$ is the identity
\begin{equation}
\label{eq:P2}
\left(\sum_{w\in W} (-1)^{\ell(w)}\, z^{2w(\rho)}\right)K(z) =1.
\end{equation}
When $Q=A_1$, this follows from a direct computation, and for arbitrary $Q$ this follows from the $A_1$-case and the Weyl denominator formula
\begin{equation}
\label{eq:Weyldenomformula}
\sum_{w\in W} (-1)^{\ell(w)}\, z^{2w(\rho)} = \prod_{\alpha \in\Delta^+} \left( z^\alpha - z^{-\alpha} \right).
\end{equation}

More generally, for $x\in W$, define the \textit{Weyl twist} of $K(z)$ by $x$ as
\begin{equation}
\label{eq:Ktwist}
K_x(z) = (-1)^{\ell(x)}\sum_{\alpha\in Q} k(\alpha)\, z^{-x(2\rho+2\alpha)}.
\end{equation}
For $x\in W$, consider the following collection of series
\begin{equation}
\label{eq:Kxn}
K_{x,n}(z):=
\left\{
\begin{array}{ll}
\left(\sum\limits_{w\in W} (-1)^{\ell(w)}\, z^{2w(\rho)}\right)^2 & \mbox{if $n=0$,}\\[12pt]
\sum\limits_{w\in W} (-1)^{\ell(w)}\, z^{2w(\rho)} & \mbox{if $n=1$,}\\[12pt]
1 & \mbox{if $n=2$,}\\[12pt]
\left(K_x(z)\right)^{n-2} & \mbox{if $n\geq 3$.}
\end{array}
\right.
\end{equation}
We will refer to this as the \textit{Kostant collection}.
The series $K_{x,n}(z)$ for $n\in\{0,1,2\}$ does not depend on $x$.

\subsection{The $q$-series}
Let $\Gamma$ be a plumbing tree, and let $M:=M(\Gamma)$ be the $3$-man\-i\-fold obtained by plumbing along $\Gamma$. After a sequence of Neumann moves, one can assume that $\Gamma$ is reduced. Consider a tuple
\begin{equation}
\label{eq:tau}
\tau=\left(Q, a, \xi\right)
\end{equation}
with
\begin{enumerate}[(i)]
\item $Q$ a root lattice;
\item $a \in {\delta} +2Q^{s}$ a representative of a $\mathrm{Spin}^c$-structure
$[a] \in \mathrm{Spin}^c_Q(M)$ as in \eqref{spinc-isomorphism}, with $s=|V(\Gamma)|$;
\item $\xi \in \Xi$ a Weyl assignment on $\Gamma$ as in \S\ref{sec:Weylassignments}.
\end{enumerate}
Define the series 
\[
\mathsf{Y}_\tau\left(q\right) = \mathsf{Y}_\tau\left(M(\Gamma);q\right) 
\]
as
\[
\mathsf{Y}_\tau\left(q \right):= 
(-1)^{|\Delta^+| \,\pi}
q^{\frac{1}{2}(3\sigma-\mathrm{tr}\,B)\langle \rho, \rho \rangle} \sum_{\ell\in a + 2BQ^s}
c_{\Gamma, \xi} (\ell)
\,q^{-\frac{1}{8}\langle \ell, \ell\rangle}
\]
where 
\begin{align}
\label{eq:cgammaxi}
c_{\Gamma, \xi}(\ell) := \prod_{v\in V(\Gamma)} \left[ K_{\xi_v,\deg v}(z_v) \right]_{\ell_v} \in \mathbb{Z}.
\end{align}
The operator $[\,\,]_\alpha$ assigns to a series in $z$ the coefficient of the monomial $z^\alpha$. 
Lastly, $\ell_v\in Q$ denotes the $v$-component of $\ell\in Q^s= Q^{V(\Gamma)}$ for $v\in V(\Gamma)$.

\smallskip

The series $\mathsf{Y}_\tau\left(q\right)$ is not always well defined. In fact, if the framing matrix $B$ is negative definite (or more generally, weakly negative definite, as in \cite[Def.~4.3]{gukov2021two}), 
then the series $\mathsf{Y}_\tau\left(q\right)$ exists and one has
\[
\mathsf{Y}_\tau\left(q\right)
\in q^{\frac{1}{2}(3\sigma-\mathrm{tr}\,B)\langle \rho, \rho \rangle -\frac{1}{8} \langle a, a \rangle} \,\mathbb{Z} 
\left(\! \left(q^{\frac{1}{2}}\right)\!\right).
\]
Indeed, in those cases, the exponents of $q$ are bounded below, and there are only finitely many $\ell$ that contribute to each power of $q$. (See also \cite[Lemma 3.2]{MT1}.) One uses that $\langle \ell, \ell \rangle = \langle a, a \rangle + 4\mathbb{Z}$ for $\ell\in a+2BQ^s$ to conclude that the powers of $q$ are half-integers, up to an overall rational shift.

A similar statement holds if $B$ is positive definite (or weakly positive definite), and in this case one replaces $q$ with $q^{-1}$, that is, $\mathsf{Y}_\tau\left(q\right)$ is Laurent in $q^{-\frac{1}{2}}$, up to an overall factor given by a rational power of $q$.

However, for an arbitrary invertible framing matrix $B$, the values $\langle \ell, \ell \rangle$ in the exponent of $q$ may not be bounded above nor below, and there might be infinitely many $\ell$ that contribute to the same value $\langle \ell, \ell \rangle$.
To overcome this issue, we introduce a new variable $t$.

\subsection{The $(q,t)$-series}
For  a reduced plumbing tree $\Gamma$ with invertible framing matrix and for $\tau=\left(Q, a, \xi\right)$ as in \eqref{eq:tau}, define the series 
\[
\mathsf{Y}_\tau\left(q,t\right) = \mathsf{Y}_\tau\left(M(\Gamma);q,t\right) 
\]
as
\[
\mathsf{Y}_\tau\left(q,t \right):= 
(-1)^{|\Delta^+| \,\pi}
q^{\frac{1}{2}(3\sigma-\mathrm{tr}\,B)\langle \rho, \rho \rangle} \sum_{\ell\in a + 2BQ^s}
c_{\Gamma, \xi} (\ell)
\,t^{\xi^{-1}(\ell)}
\,q^{-\frac{1}{8}\langle \ell, \ell\rangle}
\]
with $c_{\Gamma, \xi} (\ell)$ as in \eqref{eq:cgammaxi} 
and
\begin{equation}
\label{eq:exp_t}
\xi^{-1}(\ell): = \sum_{v} \xi^{-1}_v\left(\ell_v\right) \in Q.
\end{equation}
Thus $t^{\xi^{-1}(\ell)}$ is a multi-index monomial in variables $t_1,\dots, t_r$ as in \eqref{eq:zpowers}.

In \eqref{eq:Zdoublehat}, the series $\mathsf{Y}_\tau\left(q,t^\xi \right)$ for the variable $t^\xi$ is computed as follows:
\begin{equation}
\label{eq:txi}
(t^\xi)^{\xi^{-1}(\ell)}:= t^{\sum_v \ell_v}.
\end{equation}

\begin{lemma}
\label{lem:Yqtring}
For the series $\mathsf{Y}_\tau\left(q,t \right)$, the exponents of $t$ are bounded above, and there are only finitely many $\ell$ that contribute to each power of $t$. 
Thus $\mathsf{Y}_\tau\left(q, t\right)$ exists for all reduced plumbing trees having invertible framing matrix and is an element of  the following ring:
\[
\mathsf{Y}_\tau\left(q, t\right)
\in q^{\frac{1}{2}(3\sigma-\mathrm{tr}\,B)\langle \rho, \rho \rangle -\frac{1}{8} \langle a, a \rangle} \mathbb{Z}
 \left[q^{\pm \frac{1}{2}}\right]
\left(\! \left( t_1^{-1}, \dots, t_r^{-1} \right)\!\right).
\]
\end{lemma}

\begin{proof}
We analyze the definition in \eqref{eq:Kxn} and consider the cases when 
\begin{equation}
\label{eq:Kxnneq0}
\left[ K_{x,\deg v}(z_v) \right]_{\ell_v}\neq 0.
\end{equation}
For a vertex $v$ with $\deg v\leq 2$, there are only finitely many $\ell_v$ such that \eqref{eq:Kxnneq0} holds.
Instead for a vertex $v$ with $\deg v\geq 3$, if \eqref{eq:Kxnneq0} holds,
then necessarily $x^{-1}\left(\ell_v\right)$ is a sum of negative roots.
The statement follows.
\end{proof}

It may not be possible to evaluate the series $\mathsf{Y}_\tau\left(q,t\right)$ at $t=1$ (i.e., $t_1=\dots= t_r=1$),
as this might result in infinitely many contributions to a given monomial in~$q$. However, one has from the definition:

\begin{lemma}
If the series $\mathsf{Y}_\tau\left(q,t \right)$ can be evaluated at $t=1$, then one has 
$\mathsf{Y}_\tau\left(q,1 \right)=\mathsf{Y}_\tau\left(q\right)$.
\end{lemma}

While the series $\mathsf{Y}_\tau\left(q, t\right)$ is expressed in terms of a  plumbing presentation $\Gamma$ for $M$,  we show:

\begin{theorem}
\label{thm:qtseries}
Any two reduced plumbing trees for $M$ related by a sequence of the five Neumann moves $(A\pm), (B\pm), (C)$ yield the same series $\mathsf{Y}_\tau\left(q, t\right)$.
\end{theorem}

\subsection{Proof of invariance}
First we  prove Theorem \ref{thm:qtseriesinvarianceintroreduced}(ii) and then prove Theorem~\ref{thm:qtseries}.

\begin{remark}
\begin{enumerate}[(i)]
\item In the arguments below, we apply various identities from \cite{MT1} concerning the coefficients $c_{\Gamma, \xi}(\ell)$ from \eqref{eq:cgammaxi}. For this, we emphasize that, despite some apparent differences in the definition, these  
$c_{\Gamma,\xi}(\ell)$ are equivalent to the ones from \cite{MT1}. Indeed, the difference stems from the fact that the Weyl assignments defined here are more general than the ones used in \cite{MT1}, as their values on vertices of degree $0$ and $1$ can  possibly be arbitrary here --- see Remark \ref{rmk:Wass}. However, $K_{x,\deg v}(z)$ is independent of $x$ for $\deg v\leq 2$, as in \cite{MT1}. Hence, the coefficients $c_{\Gamma, \xi}(\ell)$ from \eqref{eq:cgammaxi} are not affected by the change of the definition of Weyl assignments.

\item The extra flexibility of the present Weyl assignments on vertices of degree $0$ and $1$ plays a role in the exponents of the variable $t$. 
\end{enumerate}
\end{remark}

\begin{theorem}
\label{thm:Winvqtseries}
For $\ell\in \delta +2Q^s$, one has
\[
c_{\Gamma, \xi} (\ell) = c_{\Gamma, w(\xi) } (w(\ell))
 \qquad \mbox{for  $w\in W$}.
\]
This implies Theorem \ref{thm:qtseriesinvarianceintroreduced}(ii).
\end{theorem}

\begin{proof}
By the definition \eqref{eq:Kxn}, one has
\[
\left[K_{x,n}(z)\right]_{\alpha} = (-1)^{\ell(w)n} \left[K_{wx,n}(z)\right]_{w(\alpha)}.
\]
Multiplying over all vertices and using the fact that the sum of the degree of the vertices is even,
the first part of the statement follows.

Since $\xi^{-1}(\ell) = (w(\xi))^{-1}(w(\ell))$ and $\langle \ell, \ell\rangle = \langle w(\ell), w(\ell)\rangle$,
one has
\begin{align*}
c_{\Gamma, \xi} (\ell)
\,t^{\xi^{-1}(\ell)} 
\,q^{-\frac{1}{8}\langle \ell, \ell\rangle}
= c_{\Gamma, w\xi } (w(\ell))
\,t^{(w(\xi))^{-1}(w(\ell))}
\,q^{-\frac{1}{8}\langle w(\ell), w(\ell)\rangle}
\end{align*}
 for  $w\in W$.
Hence the statement.
\end{proof}

We now proceed to prove Theorem \ref{thm:qtseries}.

\begin{proof}[Proof of Theorem \ref{thm:qtseries}]
We argue that the series is invariant under each of the five Neumann moves between two reduced plumbing trees.
For each move, let $\Gamma$ be the bottom plumbing tree with framing matrix $B$
and $\Gamma_\circ$ the top plumbing tree with framing matrix $B_\circ$. The number of vertices of $\Gamma$ and $\Gamma_\circ$ will be denoted by $s$ and $s_\circ$, and the number of positive eigenvalues of $B$ and $B_\circ$ will be denoted by $\pi$ and $\pi_\circ$.

For each move, we first observe how the factor in front of the sum in the series
\begin{equation}
\label{eq:frontfactor}
(-1)^{|\Delta^+| \,\pi}
q^{\frac{1}{2}(3\sigma-\mathrm{tr}\,B)\langle \rho, \rho \rangle}
\end{equation}
changes. 
We then define an injective map
\[
R\colon Q^s \rightarrow Q^{s_\circ}
\]
inducing an isomorphism of $\spinc$-structures on $M(\Gamma)$ and $M(\Gamma_\circ)$ and an isomorphism
\[
S\colon \Xi \rightarrow \Xi_\circ
\]
of Weyl assignments for $\Gamma$ and $\Gamma_\circ$.
Thus for a tuple $\tau=(Q,a,\xi)$ for $\Gamma$, we prove
\[
\mathsf{Y}_\tau \left( M(\Gamma); q,t\right) = \mathsf{Y}_{\tau_\circ} \left( M(\Gamma_\circ); q,t\right)
\]
where $\tau_\circ=(Q, R(a),S(\xi))$.
For this, we argue that the contribution of each representative $\ell$ of the $\spinc$-structure $[a]$ to the series for $\Gamma$ 
matches a sum of contributions to the series for $\Gamma_\circ$, and that all other contributions to the series for $\Gamma_\circ$ vanish.

\bigskip

\noindent \textit{Step (A$-$).} 
As shown in \cite[Proof of Thm 3.3]{MT1}, the factor \eqref{eq:frontfactor} in front of the sum in the series is invariant. 
The map $R$ for this move~is
\begin{equation*}
R\colon Q^s \rightarrow Q^{s+1}, \qquad
(a_1, a_2)\mapsto (a_1, 0,  a_2)
\end{equation*}
where the subtuple $a_1$ corresponds to the vertices of $\Gamma$ consisting of the vertex weighted by $m_1$ and all vertices on its left, and likewise $a_2$  corresponds to the vertex weighted by $m_2$ and all vertices on its right.
To define the map $S$, note that a Weyl assignment $\xi$ on $\Gamma$ uniquely determines a Weyl assignment $\xi_\circ$ on $\Gamma_\circ$ since the added vertex in $\Gamma_\circ$ has degree $2$ and all Weyl assignments assign $1_W$ to such a vertex.

From \cite[(4.3)]{MT1}, we have that
\begin{equation*}
c_{\Gamma, \xi}(\ell) \,
q^{-\frac{1}{8}\langle \ell, \ell\rangle} =
c_{\Gamma_\circ, \xi_\circ}(R(\ell))\,
q^{-\frac{1}{8}\langle R(\ell), R(\ell)\rangle},
\end{equation*}
while $c_{\Gamma_\circ, \xi_\circ}(\ell_\circ)=0$ when $\ell_\circ$ is not in the image of $R$, as we argue in \cite[Step (A$-$)]{MT1}.

Moreover, since $\Gamma$ and $\Gamma_\circ$ differ only at a single degree-$2$ vertex and the corresponding component of $R(\ell)$ is $0$ by definition, one has
\[
\xi^{-1}_\circ (R(\ell)) = \xi^{-1}(\ell)  +1_W (0)= \xi^{-1}(\ell),
\]
where $\xi^{-1}_\circ=S(\xi^{-1})$, hence the exponents of $t$ match as well. 

\bigskip

\noindent \textit{Step (A$+$).} 
In this case,  the power of $q$ in the factor \eqref{eq:frontfactor} in front of the sum in the series is invariant. 
The function $R$ for this move is
\begin{equation*}
R\colon Q^s \rightarrow Q^{s+1}, \qquad
(a_1, a_2)\mapsto (a_1, 0, -a_2)
\end{equation*}
with notation as in the previous move.
The map $S$ is $\xi\mapsto\xi_\circ$ where
for a vertex $v$ with $\deg v \neq 2$, one has
\begin{equation*}
\xi_\circ \colon v \mapsto \left\{
\begin{array}{ll}
\xi_v & \mbox{if $v$ is on the left of the added vertex,}\\[0.5pt]
\iota \xi_v & \mbox{if $v$ is on the right of the added vertex} 
\end{array}
\right.
\end{equation*}
where $\iota$ is as in \eqref{eq:iota}.

From \cite[(4.8)]{MT1} one has that
\begin{equation*}
(-1)^{|\Delta^+|\pi}\, c_{\Gamma, \xi}(\ell) \,
q^{-\frac{1}{8}\langle \ell, \ell\rangle} =
(-1)^{|\Delta^+|\pi_\circ} \,c_{\Gamma_\circ, \xi_\circ}(R(\ell))\,
q^{-\frac{1}{8}\langle R(\ell), R(\ell)\rangle},
\end{equation*}
while $c_{\Gamma_\circ, \xi_\circ}(\ell_\circ)=0$ when $\ell_\circ$ is not in the image of $R$.
As before  $\Gamma$ and $\Gamma_\circ$ differ only at a single degree-$2$ vertex.
Now, one has
\begin{align*}
\xi^{-1}_\circ(R(\ell)) 
&= \sum_{v\;\text{left} } \xi^{-1}_{\circ, v} \left(R(\ell)_v\right) +\sum_{v\;\text{right} } \xi^{-1}_{\circ, v} \left(R(\ell)_v \right)\\
&= \sum_{v\;\text{left} } \xi^{-1}_{v} \left(\ell_v \right)+\sum_{v\;\text{right} } (\iota\xi_{v})^{-1}\left(-\ell_v\right) \\
&=\xi^{-1}(\ell),
\end{align*}
since $(\iota\xi_v)^{-1}\left(-\ell_v\right)  =\xi^{-1}_v\left(\ell_v\right)$.
Thus the exponents of $t$ match. 

\bigskip

\noindent \textit{Step (B$-$).} 
Now the factor \eqref{eq:frontfactor} in front of the sum in the series for $\Gamma_\circ$ contains an extra factor $q^{-\frac{1}{2}\langle \rho, \rho\rangle}$.
The trees $\Gamma$ and $\Gamma_\circ$ differ only at a single leaf $v_0$ adjacent to a vertex $v_1$. 
For $w\in W$, consider the map from \cite[(4.9)]{MT1}
\begin{equation*}
R_w\colon Q^s \rightarrow Q^{s+1}, \qquad
(a_\sharp, a_1)\mapsto (a_\sharp, a_1+2w(\rho),  - 2 w(\rho)),
\end{equation*} 
where $a_1$ corresponds to the vertex $v_1$ of $\Gamma$, and $a_\sharp$ corresponds to all other vertices of $\Gamma$.
For the isomorphism between $\spinc$-structures, we use the map induced by $R:=R_{w}$ with $w=1_W$.

The condition \eqref{eq:forcingbridgescondition} implies that the map $S$ is given by $\xi\mapsto\xi_\circ$ 
where $\xi_{\circ, v_0} :=  \xi_{v_1}$. Moreover,  when $\deg (v_1) =1$ in $\Gamma$, thus $\deg(v_1)=2$ in $\Gamma_\circ$, we set
$\xi_{\circ, v_1}:=1_W$, otherwise
 $\xi_{\circ, v_1}:=\xi_{v_1}$; finally, $\xi_\circ$ and $\xi$ agree at all other vertices.

From \cite[(4.12)]{MT1}, we have
\begin{equation*}
c_{\Gamma, \xi}(\ell) \, q^{-\frac{1}{8}\langle \ell, \ell\rangle}
= 
q^{-\frac{1}{2}\langle \rho,\rho\rangle}
\sum_{w\in W}c_{\Gamma_\circ, \xi_\circ}(R_w(\ell))\,
q^{-\frac{1}{8}\langle R_w(\ell), R_w(\ell)\rangle},
\end{equation*}
while $c_{\Gamma_\circ, \xi_\circ}(\ell_\circ)=0$ when $\ell_\circ$ is not in the image of any $R_w$.
In the computation of the exponent of $t$, the vertices $v_0$ and $v_1$ contribute
\begin{equation}
\label{eq:xicircv1v0}
	\xi^{-1}_{\circ, v_1}\left(\ell_1 +2w(\rho)\right) + \xi^{-1}_{\circ, v_0}\left(-2w(\rho)\right) \\
	=\xi_{v_1}^{-1}\left(\ell_1\right).
\end{equation}
To show this equality, we must verify two cases. When $\deg (v_1) \neq 1$ in $\Gamma$, one has $\xi_{\circ, v_1}=\xi_{\circ, v_0}=\xi_{v_1}$ by   definition of $\xi_\circ$, and the identity follows by linearity. When $\deg (v_1) = 1$ in $\Gamma$, one has $\xi_{\circ, v_1}=1_W$ and $\xi_{\circ, v_0} =  \xi_{v_1}$. In this case, one has $c_{\Gamma_\circ, \xi_\circ}(R_w(\ell))=0$ unless $\ell_1 +2w(\rho)=0$ by \eqref{eq:Kxn} and \eqref{eq:cgammaxi}. The identity follows.
Hence $\xi^{-1}_\circ\left(R_w(\ell)\right) =\xi^{-1}\left(\ell\right)$ for all $w\in W$. 

\bigskip

\noindent \textit{Step (B$+$).} 
The factor \eqref{eq:frontfactor} in front of the sum in the series for $\Gamma_\circ$ has an extra factor $(-1)^{|\Delta^+|}q^{\frac{1}{2}\langle \rho, \rho\rangle}$.
As with the previous move, $\Gamma$ and $\Gamma_\circ$ differ only at a single leaf $v_0$ adjacent to a common vertex $v_1$. 
For $w\in W$, consider the map from \cite[(4.13)]{MT1} 
\begin{equation*}
R_w\colon Q^s \rightarrow Q^{s+1}, \qquad
(a_\sharp, a_1)\mapsto (a_\sharp, a_1+2w(\rho),  2 w(\rho)).
\end{equation*}
As before, $R=R_{w}$ with $w=1_W$.

By \eqref{eq:forcingbridgescondition}, the map $S$ is given by $\xi\mapsto\xi_\circ$ 
where $\xi_{\circ, v_0} :=  \iota\xi_{v_1}$. Moreover,  when $\deg (v_1) =1$ in $\Gamma$, thus $\deg(v_1)=2$ in $\Gamma_\circ$, we set
$\xi_{\circ, v_1}:=1_W$, otherwise
 $\xi_{\circ, v_1}:=\xi_{v_1}$; finally, $\xi_\circ$ and $\xi$ agree at all other vertices.

From \cite[(4.16)]{MT1}, we have
\begin{align*}
&(-1)^{|\Delta^+|\, \pi}
c_{\Gamma, \xi}(\ell) \, q^{-\frac{1}{8}\langle \ell, \ell\rangle}\\
&= 
(-1)^{|\Delta^+|\, \pi_\circ}
q^{\frac{1}{2}\langle \rho,\rho\rangle}
\sum_{w\in W}c_{\Gamma_\circ, \xi_\circ}(R_w(\ell))\,
q^{-\frac{1}{8}\langle R_w(\ell), R_w(\ell)\rangle},
\end{align*}
while $c_{\Gamma_\circ, \xi_\circ}(\ell_\circ)=0$ when $\ell_\circ$ is not in the image of any $R_w$.
Moreover,
\[
	\xi^{-1}_{\circ, v_1}\left(\ell_1 +2w(\rho)\right) + \xi^{-1}_{\circ, v_0}\left(2w(\rho)\right) 
	=\xi^{-1}_{v_1}\left(\ell_1\right).
\]
This follows similarly to \eqref{eq:xicircv1v0} using  
\[
(\iota\xi_{v_1})^{-1}\left(2w(\rho)\right) = \xi^{-1}_{v_1}\left(-2w(\rho)\right). 
\]
Hence \sloppy{$\xi^{-1}_\circ\left(R_w(\ell)\right) =\xi^{-1}\left(\ell\right)$} for all $w\in W$. 

\bigskip

\noindent \textit{Step (C).} 
The factor \eqref{eq:frontfactor}  in front of the sum in the series for $\Gamma_\circ$ has an extra factor $(-1)^{|\Delta^+|}$.
Let $v_0$ be the vertex of $\Gamma$ with weight $m_1+m_2$, and let $v_1, v'_0, v_2$ be the vertices of $\Gamma_\circ$ with weights  $m_1, 0$, and $m_2$, respectively.
For $a\in Q^s$, write $a=(a_\sharp, a_0, a_\flat)$, where the entry $a_0$ corresponds to $v_0$, the subtuple $a_\sharp$ corresponds to all vertices on the left of $v_0$, and the subtuple $a_\flat$ to all vertices on the right of $v_0$.
For $\beta\in Q$, consider the map from \cite[(1.8)]{MT1}
\[
R_\beta\colon Q^s \rightarrow Q^{s+2}, \qquad
(a_\sharp, a_0, a_\flat)\mapsto (a_\sharp, a_0+\beta,0,\beta, -a_\flat)
\]
where the entries $a_0 + \beta, 0$, and $\beta$ correspond to the vertices $v_1, v'_0, v_2$ in $\Gamma_\circ$.
For $\beta=\beta_0$ defined as in \cite[(1.9)]{MT1}, the map $R=R_\beta$ induces an isomorphism of $\spinc$-structures.

The map $S$ is defined as $\xi\mapsto\xi_\circ$ such that, 
for a vertex $v$ with $\deg v \neq 2$, one has
\begin{equation*}
\xi_\circ \colon v \mapsto \left\{
\begin{array}{ll}
\xi_v & \mbox{if $v<v_1$,}\\[0.5pt]
\iota\xi_v & \mbox{if $v>v_2$.} 
\end{array}
\right.
\end{equation*}
Here $v<v'$ if $v$ is on the left of $v'$.
Moreover, define

\[
\xi_{\circ, v_1} := \left\{
\begin{array}{ll}
\xi_{v_0} & \mbox{if $\deg(v_1)\neq 2$,}\\[5pt]
1_W & \mbox{if $\deg(v_1)= 2$,}
\end{array}
\right.
\]
and
\[
\xi_{\circ, v_2} :=  \left\{
\begin{array}{ll}
\iota\xi_{ v_0} & \mbox{if $\deg(v_2)\neq 2$,}\\[5pt]
1_W & \mbox{if $\deg(v_2)= 2$.}
\end{array}
\right.
\]
From \cite[(4.22)]{MT1}, one has
\[
(-1)^{|\Delta^+|\, \pi}
c_{\Gamma,\xi}(\ell) 
q^{-\frac{1}{8}\langle \ell, \ell\rangle}=  
(-1)^{|\Delta^+|\, \pi_\circ} \sum_{\beta\in \beta_0+2Q}c_{\Gamma_\circ,\xi_\circ}(R_\beta(\ell))
q^{-\frac{1}{8}\langle R_\beta(\ell), R_\beta(\ell)\rangle},
\]
while $c_{\Gamma_\circ, \xi_\circ}(\ell_\circ)=0$ when $\ell_\circ$ is not in the image of any $R_\beta$.
Furthermore, for each $\beta$, one has
\begin{align*}
\xi^{-1}_\circ\left(R_\beta(\ell)\right) 
&= \sum_{v<v_1} \xi^{-1}_{\circ, v} \left(\ell_v \right)
+\xi^{-1}_{\circ, {v_1}}\left(\ell_0+\beta\right)
+\xi^{-1}_{\circ, {v_2}}\left(\beta\right)
+\sum_{v>v_2} \xi^{-1}_{\circ, v} \left( -\ell_v\right) \\
&= \sum_{v<v_0} \xi^{-1}_{v} \left(\ell_v \right)
+\xi^{-1}_{v_0}\left(\ell_0+\beta\right)
+(\iota\xi_{v_0})^{-1}\left( \beta \right)
+\sum_{v>v_0} (\iota\xi_{v})^{-1} \left(-\ell_v\right) \\
&= \sum_{v<v_0} \xi^{-1}_{v} \left(\ell_v \right)+ \xi^{-1}_{v_0}\left(\ell_{v_0}\right) +\sum_{v>v_0} \xi^{-1}_{v}\left(\ell_v\right) \\
&=\xi^{-1}(\ell),
\end{align*}
since $(\iota\xi_v)^{-1}(-\ell_v)  =\xi^{-1}_v(\ell_v)$.
Thus $\xi^{-1}_\circ(R_\beta(\ell)) =\xi^{-1}(\ell)$ for all $\beta\in Q$. 

\smallskip

This concludes the proof of the statement.
\end{proof}

%%%%%%%%%%%%%%%%%%%%%%%%%%%%
%%%%%%%%%%%%%%%%%%%%%%%%%%%%
%%%%%%%%%%%%%%%%%%%%%%%%%%%%

\section{An invariant three-variable series for knot complements}
\label{sec:knotinvt}

Here we define a $(q,t,z)$-series for plumbed knot complements and prove Theorem \ref{thm:qtzseriesinvarianceintroreduced}.
We use notation as in \S\S\ref{sec:neumann-knots}--\ref{sec:relspinc}.

Consider a plumbed knot complement $(M, \partial M)$. We may assume that after a sequence of Neumann moves on pairs, $(M, \partial M)$ is constructed by plumbing along a reduced pair $(\Gamma, v_0)$ as in \S\ref{sec:redplpairs}. We assume that $\Gamma$ has invertible framing matrix.

Consider a tuple $\tau=\left(Q,a,\xi \right)$ as in \eqref{eq:tau}.
We define the series
\[
\mathsf{Y}_\tau\left(q, t, z\right) :=
\mathsf{Y}_\tau\left(M(\Gamma, v_0); q, t, z\right)
\]
as
\begin{multline}
\label{eq:Ypa}
\mathsf{Y}_\tau\left(q, t, z\right)\\
:= 
(-1)^{|\Delta^+| \,\pi}
q^{\frac{1}{2}(3\sigma-\mathrm{tr}\,B)\langle \rho, \rho \rangle} 
\sum_{\ell\in a + 2Q\langle B_1, \dots, B_s\rangle}
c_{\Gamma, \xi, v_0} (\ell) 
\, t^{\xi^{-1}(\ell)}
\,q^{-\frac{1}{8}\langle \ell, \ell\rangle}
\end{multline}
where 
\begin{equation}
\label{eq:cGammaell}
c_{\Gamma, \xi, v_0} (\ell)  
:= z^{-\ell_{v_0}}  K_{\xi(v_0),1+\deg v_0}(z) \prod_{v\neq v_0} \left[ K_{\xi_v, \deg v}(z_v) \right]_{\ell_v}.
\end{equation}
Here the notation is as in \eqref{eq:cgammaxi}. 
The coefficients $c_{\Gamma, \xi, v_0} (\ell)$ lie in a ring that depends on $\xi(v_0)$. E.g., for $\xi(v_0)=1_W$, since 
$K(z)\in \mathbb{Z}\left\llbracket z_1^{- 1}, \dots, z_r^{- 1}\right\rrbracket$, one has
\[
c_{\Gamma, \xi, v_0} (\ell) \in \mathbb{Z}\left( \!\left( z_1^{- 1}, \dots, z_r^{- 1}\right)\!\right)
\]
and thus as in Lemma \ref{lem:Yqtring}, one has
\[
\mathsf{Y}_\tau\left(q, t, z\right) \in 
q^{\frac{1}{2}(3\sigma-\mathrm{tr}\,B)\langle \rho, \rho \rangle -\frac{1}{8} \langle a, a \rangle} 
\mathbb{Z}\left( \!\left( z_1^{- 1}, \dots, z_r^{- 1}\right)\!\right)
 \left[q^{\pm \frac{1}{2}}\right]
\left(\! \left( t_1^{-1}, \dots, t_r^{-1} \right)\!\right).
\]

\begin{proof}[Proof of Theorem \ref{thm:qtzseriesinvarianceintroreduced}]
One needs to check invariance under the five Neumann moves given in Figure \ref{fig:Neumann} for which  the distinguished vertex $v_0$ is not one of the vertices weighted by $\pm 1$ or $0$ in the top plumbing trees there (see \S\ref{sec:neumann-knots}).
For this, the argument for the proof of Theorem \ref{thm:qtseries} applies after replacing the contribution 
\[
\left[ K_{\xi(v_0), \deg v_0}(z_{v_0}) \right]_{\ell_{v_0}}
\]
to the coefficients $c_{\Gamma, \xi} (\ell)$ in \eqref{eq:cgammaxi} with the contribution 
\[
z^{-\ell_{v_0}}  K_{\xi(v_0),1+\deg v_0}(z)
\]
to $c_{\Gamma, \xi, v_0} (\ell)$ in \eqref{eq:cGammaell}.
\end{proof}

%%%%%%%%%%%%%%%%%%%%%%%%%%%%
%%%%%%%%%%%%%%%%%%%%%%%%%%%%
%%%%%%%%%%%%%%%%%%%%%%%%%%%%

\section{A gluing formula}
\label{sec:gluingformula}

We consider here a closed oriented 3-manifold $M$ obtained by gluing two plumbed knot complements and show how to obtain the $(q,t)$-series of $M$ from the $(q,t,z)$-series of the knot complements via a gluing formula.
We use notation as in \S\S\ref{sec:gluing}--\ref{subsec:spinc-gluing}.

Assume $M=M(\Gamma)$ is obtained by gluing a pair of plumbed knot complements  
\[
\left(M^\pm, \partial M^\pm\right)=M\left(\Gamma^\pm, v^\pm_0\right) 
\] 
along their boundaries. 
Assume that $\Gamma$ is reduced, the pairs $\left(\Gamma^\pm, v^\pm_0\right)$ are reduced, and $\Gamma$ and $\Gamma^\pm$ have invertible framing matrices.
Consider a tuple $\tau=(Q,a,\xi)$ for $M$ as in \eqref{eq:tau}. Starting from $a$, select
representatives
\[
{a}^+\in \widehat{\delta}^+ +2Q^{m} \qquad\mbox{and}\qquad {a}^-\in \widehat{\delta}^- +2Q^{n}
\]
of relative $\mathrm{Spin}^c$-structures $[a^\pm] \in \mathrm{Spin}^c_Q\left(M^\pm, \partial M^\pm\right)$ as in \eqref{eq:idspincgluing} 
such that \mbox{$a=a^+ \ast a^-$.} 
This condition implies that the isomorphism of $\mathrm{Spin}^c$-structures in \eqref{eq:spinciso} identifies $[a]$ with the orbit of $[a^+]\oplus[a^-]$ under the action of 
\[
H_1(\partial M^+;Q)\cong Q\langle\lambda, \mu\rangle.
\] 
Starting from $\xi$, define $\xi^\pm$ to be the Weyl assignments on $\Gamma^\pm$ given by restricting $\xi$.

\begin{theorem}
\label{thm:gluing}
One has
\begin{align}
\label{eq:gluid}
\mathsf{Y}_\tau\left(M; q, t\right) = 
(-1)^\triangle q^\square  
\sum_{\gamma \in Q}
 \left[\mathsf{Y}^+_{\gamma}(z)  \, \mathsf{Y}^-_{\gamma}(z)  \right]_0
\end{align}
where
\begin{align}
\begin{split}
\label{eq:trianglesquare}
	\triangle &:= |\Delta^+| \left( \pi(B)- \pi(B^+) -\pi(B^-) \right),\\
	\square &:= \frac{3}{2}\left(\sigma(B) - \sigma(B^+) - \sigma(B^-) \right)\langle \rho, \rho \rangle\\
	&\qquad-\frac{1}{8}\langle a,a\rangle +\frac{1}{8}\langle a^+, a^+\rangle +\frac{1}{8}\langle a^-, a^-\rangle,\\
	\mathsf{Y}^{\pm}_{\gamma}(z)&:= \mathsf{Y}_{\tau^\pm}\left(M^\pm; q, t, z\right),
\end{split}
\end{align}
with tuples
\[
\tau^\pm = \tau^\pm(\gamma) := \left(Q, b^\pm, \xi^\pm \right) \qquad \mbox{for $\gamma\in Q$,}
\]
and
\begin{align}
\label{eq:bpm}
b^+ =b^+(\gamma)&:= a^+ + 2 \gamma B^+_m, &
b^-=b^-(\gamma)&:=a^- + 2 \gamma B^-_1.
\end{align}
\end{theorem}

\begin{remark}
\label{rmk:orbitla}
The orbit of $[a^+]\oplus[a^-]$ under the action of $Q\langle\lambda\rangle$ from \eqref{eq:lamuaction} is
\begin{align*}
&Q\langle\lambda\rangle \left([a^+]\oplus[a^-]\right)\\
&=\left\{ [b^+]\oplus[b^-] : b^+= b^+(\gamma) \mbox{ and } b^-=b^-(\gamma) \mbox{ for some }\gamma\in Q\right\}.
\end{align*}
Hence, under the isomorphism \eqref{eq:spinciso}, all $[b^+]\oplus[b^-]$ for $\gamma \in Q$ are identified with the same $\spinc$-structure $[a]$ on $M$.
Indeed, the isomorphism \eqref{eq:spinciso} maps $[b^+]\oplus[b^-]$ to the class of
\[
b^+ \ast b^- = a^+ \ast a^- + 2\gamma\left(B^+_m \ast B^-_1 \right) \qquad \mbox{for $\gamma\in Q$}. 
\]
As $B^+_{m} \ast B^-_1 = B_{m}$, one has 
\[
b^+ \ast b^- \equiv a^+ \ast a^- \mbox{ mod }2Q\langle B_1, \dots, B_{s} \rangle.
\]
Hence, all $b^+ \ast b^-$ for $\gamma \in Q$ represent the same class $[a]$ on $M$.
\end{remark}

\begin{lemma}
\begin{enumerate}[(i)]

\item The quantity $\square$ is independent of the choice of representatives of $\mathrm{Spin}^c$-structures $a^\pm\in[a^\pm]$ and $a\in [a]$ subject to the constraint $a=a^+\ast a^-$, and so is the rest of the right-hand side of \eqref{eq:gluid}.

\item Moreover, $\square$ is invariant under the action of $Q\langle\lambda\rangle$ on $[a^+]\oplus[a^-]$, and so is the rest of the right-hand side of \eqref{eq:gluid}.

\item While $\square$ is not invariant under the action of $Q\langle\mu\rangle$ on $[a^+]\oplus[a^-]$, the whole right-hand side of \eqref{eq:gluid} is.
\end{enumerate}
\end{lemma}

\begin{proof}
For (i), consider $c^\pm \in \left[a^\pm\right]$, that is,
\begin{align*}
c^+ &= a^+ + 2 B^+ v^+ \qquad \mbox{for some $v^+\in Q^{m-1}\times \{0\}$,}\\
c^- &= a^- + 2 B^- v^- \qquad \mbox{for some $v^-\in \{0\}\times Q^{n-1}$.}
\end{align*}
Let $c:=c^+ \ast c^-$. Then  $c = a + 2B\left( v^+ \ast v^- \right)$, hence $c\in [a]$.
To verify the statement about $\square$ it is enough verify that
\[
\langle a,a\rangle -\langle a^+, a^+\rangle -\langle a^-, a^-\rangle
=\langle c,c\rangle -\langle c^+, c^+\rangle -\langle c^-, c^-\rangle.
\]
Expanding, one has
\begin{align}
\begin{split}
\label{eq:expaacc}
\langle c,c\rangle &= \langle a,a\rangle + 4 a^T \left(  v^+ \ast v^- \right) + 4 \left( v^+ \ast v^-\right)^T B \left( v^+ \ast v^-\right),\\
\langle c^\pm,c^\pm\rangle &= \langle a^\pm,a^\pm\rangle + 4 \left(a^\pm\right)^T  v^\pm + 4 \left( v^\pm\right)^T B^\pm v^\pm.
\end{split}
\end{align}
The statement follows from the identities
\begin{align*}
a^T \left(  v^+ \ast v^- \right) &= \left(a^+\right)^T  v^+ +  \left(a^-\right)^T  v^-,\\
\left( v^+ \ast v^-\right)^T B \left( v^+ \ast v^-\right) &= \left( v^+\right)^T B^+ v^+ + \left( v^-\right)^T B^- v^-
\end{align*}
which hold by linearity.

\smallskip

For (ii), consider $b^\pm$ in the orbit of $[a^+]\oplus[a^-]$ under the action of $Q\langle\lambda\rangle$, that is, $b^\pm = b^\pm (\gamma)$ for some $\gamma\in Q$ as in \eqref{eq:bpm}, see Remark \ref{rmk:orbitla}.
Let $b:=b^+\ast b^-$. Then $b=a+2\gamma B_m$.
To verify the statement about $\square$, it is enough verify that
\[
\langle a,a\rangle -\langle a^+, a^+\rangle -\langle a^-, a^-\rangle
=\langle b,b\rangle -\langle b^+, b^+\rangle -\langle b^-, b^-\rangle.
\]
Write 
\begin{align*}
b^+ &= a^+ +2\gamma B^+ e_m, &
b^- &= a^- +2\gamma B^- e_1, &
b &= a +2\gamma B e_m,
\end{align*}
where $e_i$ is the $i$-th standard basis element in the appropriate vector space.
Expanding as in \eqref{eq:expaacc}, the statement follows from the identities

\begin{align*}
a^T e_m &= \left(a^+\right)^T e_m +  \left(a^-\right)^T  e_1,\\
e_m^T B  e_m &= e_m^T B^+ e_m + e_1^T B^- e_1
\end{align*}
which hold by definition of the operation $\ast$ in \eqref{ast} and \eqref{eq:Bast}.

\smallskip

For (iii), consider $d^\pm$ in the orbit of $[a^+]\oplus[a^-]$ under the action of $Q\langle\mu\rangle$, that is,
$d^+ = a^+ + 2\delta e_m$ and $d^- = a^- -2\delta e_1$ for some $\delta\in Q$.
One has $d^+ \ast d^- = a$. A direct computation as in parts (i)-(ii) shows that 
\[
\langle a^+, a^+\rangle +\langle a^-, a^-\rangle \neq
\langle d^+, d^+\rangle +\langle d^-, d^-\rangle
\]
hence $\square$ is not invariant under the action of $Q\langle\mu\rangle$.
One can directly check the invariance of the whole right-hand side of \eqref{eq:gluid}. This will also follow from the proof of Theorem \ref{thm:gluing} below.
\end{proof}

\begin{proof}[Proof of Theorem \ref{thm:gluing}]
On each side of the identity \eqref{eq:gluid}, the series is obtained as a sum of contributions indexed by the representatives of the $\spinc$-structure. First, we give a bijection between these representatives, and then we argue that the corresponding contributions coincide.

The set of representatives of the $\spinc$-structure on the left-hand side is $a+2Q\langle B_1, \dots, B_{s} \rangle$. Select 
 such an element
\[
\ell = a+2Bv \qquad\mbox{for some } v\in Q^s.
\]
On the right-hand side, the set of representatives of the $\spinc$-structure is
indexed by $\gamma\in Q$ and representatives of the classes $[b^+]$ and $[b^-]$, that is, the set of representatives is
\[
\bigcup_{\gamma\in Q} \left( a^+ + 2 \gamma B^+_m + 2Q\langle B^+_1, \dots, B^+_{m-1} \rangle \right)\times 
\left( a^- + 2 \gamma B^-_1 + 2Q\langle B^-_2, \dots, B^-_{n} \rangle \right).
\]
We assign to $\ell$ an element of this set as follows.
Write $v=\left(v_1, \dots, v_s \right)$, and let $\gamma:=v_m$  be the $m$-th coordinate of $v$.
Define 
\begin{align}
\label{eq:vpm}
v^+ &:=\left(v_1, \dots, v_{m-1}, \gamma\right), & 
v^- &:=\left(\gamma, v_{m+1}, \dots, v_s\right).
\end{align}
One has
$v=v^+\ast v^- -\gamma e_m$ where $e_m$ is the $m$-th standard basis element, that is,
\begin{equation}
\label{eq:vvplusvminus}
v= \left( v^+_1, \dots, v^+_{m-1}, \gamma, v^-_2, \dots, v^-_n \right).
\end{equation}
Define 
\[
\ell^\pm:=a^\pm +2B^\pm v^\pm \in \left[b^\pm \right].
\]
This gives a map $\ell \mapsto (\gamma, \ell^+, \ell^-)$. 
One has $\ell =\ell^+ \ast \ell^-$. This follows from
\[
\ell^+\ast \ell^- = a^+ \ast a^- + 2 B^+ v^+ \ast B^-v^- = a + 2Bv = \ell.
\]
Vice versa, the map  
\[
(\gamma,\ell^+, \ell^-)\mapsto\ell:=\ell^+ \ast \ell^-
\]
is clearly the desired inverse.

Next, we compare the exponents of $q$. Since $\ell=a+2Bv$, one has
\begin{align}
\langle \ell, \ell\rangle &= \langle a, a\rangle + 4 \langle a, Bv\rangle + 4\langle Bv, Bv\rangle\nonumber\\
\label{eq:ellell}
&= \langle a, a\rangle + 4 a^T v + 4v^T Bv.
\end{align}
Similarly, since $\ell^\pm=a^\pm+2B^\pm v^\pm$, one has
\[
\langle \ell^\pm, \ell^\pm\rangle = \langle a^\pm, a^\pm\rangle + 4 \left(a^\pm\right)^T v^\pm + 4\left(v^\pm\right)^T B^\pm v^\pm.
\]
Since $a=a^+\ast a^-$ and using \eqref{eq:vvplusvminus}, one directly verifies that
\begin{equation}
\label{eq:astdotprop}
a^T v = \left(a^+\right)^T v^+ + \left(a^-\right)^T v^-.
\end{equation}
Similarly, using \eqref{eq:Bast}, one has
\begin{equation}
\label{eq:astBdotprop}
v^T Bv = \left(v^+\right)^T B^+ v^+ + \left(v^-\right)^T B^- v^-.
\end{equation}
Replacing \eqref{eq:astdotprop} and \eqref{eq:astBdotprop} in \eqref{eq:ellell} and simplifying, one obtains
\[
\langle \ell, \ell\rangle = \langle a,a\rangle - \langle a^+, a^+\rangle - \langle a^-, a^-\rangle + \langle \ell^+, \ell^+\rangle + \langle \ell^-, \ell^-\rangle.
\]
This together with the fact that $\mathrm{tr} \, B=\mathrm{tr} \, B^+ + \mathrm{tr} \, B^-$ implies that the exponent of $q$ for the contribution given by $\ell$ to the left-hand side matches the exponent of $q$ for the contribution given by $(\gamma, \ell^+, \ell^-)$ to the right-hand~side.

The matching of the exponents of $t$ is equivalent to the identity
\[
\xi^{-1}(\ell) = (\xi^+)^{-1}(\ell^+) + (\xi^-)^{-1}(\ell^-)
\]
which follows from $\ell=\ell^+\ast\ell^-$ and the fact that $\xi$ and $\xi^\pm$ all have equal value at the vertex $v_0$.
The matching of the exponents of $u$ follows similarly.

Finally, we compare the coefficients of the two sides. 
From the definition \eqref{eq:Kxn}, one has
\[
K_{\xi(v_0),\delta_0}(z) = K_{\xi(v_0),1+\delta^+_0}(z)\, K_{\xi(v_0),1+\delta^-_0}(z)
\]
where $\delta^\pm_0$ is the degree of $v^\pm_0$ in $\Gamma^\pm$ and $\delta_0 = \delta^+_0 + \delta^-_0$ is the degree of $v_0$ in $\Gamma$. 
Since $\ell= \ell^+ \ast \ell^-$, one has $\ell_0 = \ell_0^+ + \ell_0^-$, where $\ell_0$ is the component of $\ell$ corresponding to $v_0$, and similarly $\ell_0^\pm$ is the component of $\ell^\pm$ corresponding to $v_0^\pm$.
This implies
\begin{equation}
\label{eq:PwlPplusPminus}
\left[ K_{\xi(v_0),\delta_0}(z) \right]_{\ell_0} = \left[ z^{-\ell^+_0}K_{\xi(v_0),1+\delta^+_0}(z) \, z^{-\ell^-_0}K_{\xi(v_0),1+\delta^-_0}(z) \right]_0.
\end{equation}
The contribution given by vertices $v\neq v_0$ to the two sides equals
\[
\prod_{v\neq v_0} \left[ K_{\xi_v,\deg v}(z_v) \right]_{\ell_v}.
\]
Multiplying this on both sides of \eqref{eq:PwlPplusPminus} 
yields the equality of the coefficients 
\[
c_{\Gamma, \xi} (\ell) = \left[c_{\Gamma^+, \xi^+, v_0} (\ell^+) \, c_{\Gamma^-, \xi^-, v_0} (\ell^-) \right]_0
\]
on the two sides of the identity, 
hence the statement.
\end{proof}

%%%%%%%%%%%%%%%%%%%%%%%%%%%%%
%%%%%%%%%%%%%%%%%%%%%%%%%%%%%
%%%%%%%%%%%%%%%%%%%%%%%%%%%%%

\section{On the splitting move}

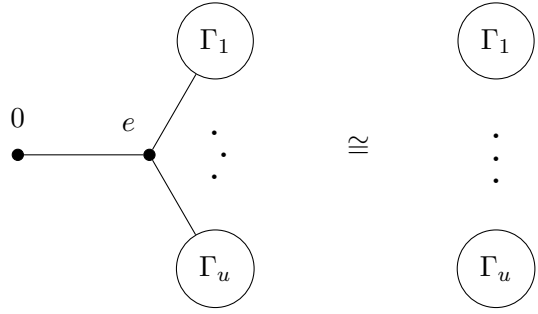
\begin{figure}[t]
\hspace{-0.8cm}
\[
\begin{tikzpicture}[baseline={([yshift=0ex]current bounding box.center)}]
      \path(0,0) ellipse (5 and 5);
      \tikzstyle{level 1}=[counterclockwise from=-90,level distance=15mm,sibling angle=120]
      \node [draw,circle,inner sep=5] (A1) at (60:5) {${\Gamma_1}$};
      \tikzstyle{level 1}=[counterclockwise from=-90,level distance=15mm,sibling angle=120]
      \node [draw,circle,inner sep=5] (A2) at (-60:5) {${\Gamma_u}$};
      \tikzstyle{level 1}=[counterclockwise from=180,level distance=12mm,sibling angle=60]
      \node [draw,circle,fill, inner sep=1.5, label={[label distance=8]120:$e$}] (A0) at (0:0) {}
      	    child [grow=45] {node[draw=none, right=0.5] {$\textbf{.}$} edge from parent[draw=none]}
	    child [grow=0] {node[draw=none, right=0.5] {$\textbf{.}$} edge from parent[draw=none]}
	    child [grow=-45] {node[draw=none, right=0.5] {$\textbf{.}$} edge from parent[draw=none]};
      \tikzstyle{level 1}=[counterclockwise from=180,level distance=12mm,sibling angle=60]
      \node [draw,circle,fill, inner sep=1.5, label={[label distance=8]90:$0$}] (A3) at (180:5) {};      
	\path (A0) edge []  node[auto, above=0.5, label={180:}]{}
      				  node[near end, below=0.1, label={-90:}]{} (A1);
	\path (A0) edge []  node[auto, above=0.5, label={180:}]{}
      				  node[near end, below=0.1, label={-90:}]{} (A2);
	\path (A0) edge []  node[auto, above=0.5, label={180:}]{}
      				  node[near end, below=0.1, label={-90:}]{} (A3);				  
    \end{tikzpicture}
\qquad    \cong 
\begin{tikzpicture}[baseline={([yshift=0ex]current bounding box.center)}]
      \path(0,0) ellipse (2 and 5);
      \tikzstyle{level 1}=[counterclockwise from=-90,level distance=15mm,sibling angle=120]
      \node [draw,circle,inner sep=5] (A1) at (60:5) {${\Gamma_1}$}
      	    child [grow=-90] {node[draw=none, below=1] {$\textbf{.}$} edge from parent[draw=none]}
	    child [grow=-90] {node[draw=none, below=0.7] {$\textbf{.}$} edge from parent[draw=none]}
	    child [grow=-90] {node[draw=none, below=1.3] {$\textbf{.}$} edge from parent[draw=none]};
      \tikzstyle{level 1}=[counterclockwise from=-90,level distance=15mm,sibling angle=120]
      \node [draw,circle,inner sep=5] (A2) at (-60:5) {${\Gamma_u}$};
    \end{tikzpicture}
\]
\caption{Neumann's move ($\mathrm{D}_{e,u}$).}
\label{fig:splitting}
\end{figure}

Neumann showed that if two plumbed $3$-manifolds obtained from two forests $F_1$ and $F_2$ admit an orientation-preserving diffeomorphism, then $F_1$ and $F_2$ are related by a sequence of the moves from Figure \ref{fig:Neumann}
together with the splitting move ($\mathrm{D}_{e,u}$) in Figure \ref{fig:splitting} \cite{neumann2006invariant}.
There the weight $e$ is an arbitrary integer and $u\geq 1$. 
For $u\geq 2$, this move relates a tree with a disjoint union of trees $\Gamma_1, \dots, \Gamma_u$.

For $Q=A_1$, Ri showed that his $(q,t)$-series is not invariant under the splitting move \cite{ri2023refined}.
Here, we make explicit how our $(q,t)$-series in the case of arbitrary root lattices varies under the splitting move.

By Remark \ref{rmk:Du>=3} at the end of this section, it is enough to consider the case $u=2$.
Let $\Gamma_\circ$ be the plumbing tree on the left-hand side in Figure \ref{fig:splitting}. 
We show that the $(q,t)$-series for $\Gamma_\circ$ decomposes as a sum of products of certain restrictions of the $(q,t)$-series for $\Gamma_1$ and $\Gamma_2$ times an additional $(q,t)$-series. 

Let $v_0$ and $v_e$ be the vertices of $\Gamma_\circ$ weighted by $0$ and $e$, respectively. For $i\in\{1,2\}$, let $v^i_{*}$ be the vertex of $\Gamma_i$ incident to $v_e$ in $\Gamma_\circ$. The degree of $v^i_{*}$ in $\Gamma_\circ$ is one more than  its degree in $\Gamma_i$.

%After possibly applying the Neumann move (A$-$) to $\Gamma_\circ$ along the edge incident to $v_e$ and $v^i_*$  and correspondingly the Neumann move (B$-$) to $\Gamma_i$, we can assume without loss of generality that $v^i_*$ has degree $1$ in $\Gamma_i$ and thus degree $2$ in $\Gamma_\circ$. 

Assume that $\Gamma_\circ$, $\Gamma_1$ and $\Gamma_2$ are reduced and that $\Gamma_1$ and $\Gamma_2$ have invertible framing matrices. Then necessarily $\Gamma_\circ$ has invertible framing matrix as well --- see the following \eqref{eq:sigmasplitting}.
Moreover, we proceed under the assumption that $v^i_*$ has degree $1$ in $\Gamma_i$ and thus degree $2$ in $\Gamma_\circ$.
\footnote{In principle, one could apply the Neumann move (A$-$) to $\Gamma_\circ$ along the edge incident to $v_e$ and $v^i_*$,  and correspondingly (B$-$) to $\Gamma_i$ (or the Neumann moves (A$+$) and (B$+$), respectively) in order to satisfy this degree assumption. However, in some cases these Neumann moves could introduce a reducible vertex in $\Gamma_i$. Thus, this assumption  on the degrees is restrictive. A similar assumption was also used in the splitting formula in \cite{ri2023refined}.}

Let $\tau_\circ=(Q, a^\circ, \xi_\circ)$ be a tuple for $\Gamma_\circ$ as in \eqref{eq:tau}.
Write $a^\circ$ as 
\[
a^\circ =\left(a^\circ_0, a^\circ_e, a^1_*-2\rho, a^1_\sharp,  a^2_*-2\rho, a^2_\sharp  \right).
\]
Here the  entry $a^\circ_0$ corresponds to $v_0$, the entry $a^\circ_e$ to $v_e$, the entries $a^i_*-2\rho$ to $v^i_{*}$, and the subtuple $a^i_\sharp$ to the remaining vertices of $\Gamma_i$. The map
\[
a^\circ \mapsto \left( a^1,  a^2\right)
\quad \mbox{with}\quad
a^i:=\left( a^i_*, a^i_\sharp \right)
\]
induces an isomorphism of the spaces of $\spinc$-structures 
\[
\spinc_Q\left(\Gamma_\circ\right)
\rightarrow 
\spinc_Q\left(\Gamma_1\right)\times \spinc_Q\left(\Gamma_2\right) 
\qquad 
\left[a^\circ\right] \mapsto
 \left[\left(a^1, a^2\right)\right]  .
\]

Next, we focus on the Weyl assignments.  
The Weyl assignment $\xi_\circ$ on $\Gamma_\circ$ uniquely determines Weyl assignments $\xi_ i$ on $\Gamma_i$ such that 
$\xi_\circ$ and $\xi_i$ have the same values on the vertices of $\Gamma_i$ other than $v^i_*$ for $i\in\{1,2\}$.
Since $v^i_*$ has degree $2$ in $\Gamma_\circ$, one has necessarily $\xi_\circ(v^i_*)=1_W$. 
Define $\tau_i:=(Q, a^i, \xi_i)$ for $i\in\{1,2\}$.

Next, for $w\in W$ and a plumbing tree $\Gamma$ whose vertices are ordered so that the first vertex is a leaf,
 define a restricted $(q,t)$-series for $\Gamma$ as
\[
\mathsf{Y}^w_{\tau}\left(\Gamma; q,t\right) :=
(-1)^{|\Delta^+| \,\pi}
q^{\frac{1}{2}(3\sigma-\mathrm{tr}\,B)\langle \rho, \rho \rangle} \sum_{\ell\in R^w}
c_{\Gamma, \xi} (\ell)
\,t^{\xi^{-1}(\ell)}
\,q^{-\frac{1}{8}\langle \ell, \ell\rangle}
\]
where
\[
R^w:=\left(a + 2BQ^s\right) \cap \left(\{2w(\rho)\}\times Q^{s-1}\right).
\]
In other words, 
the series $\mathsf{Y}^w_{\tau}\left(\Gamma; q,t\right)$ is obtained by restricting the sum in the series
$\mathsf{Y}_{\tau}\left(\Gamma; q,t\right)$ over only those $\ell\in a + 2BQ^s$ whose first entry (which corresponds to a leaf) is fixed equal to $2w(\rho)$. 

By definition, one has
\[
\mathsf{Y}_{\tau}\left(\Gamma; q,t\right) 
= \sum_{w\in W}
\mathsf{Y}^w_{\tau}\left(\Gamma; q,t\right).
\]
The various series $\mathsf{Y}^w_{\tau}\left(\Gamma; q,t\right)$ are not separately invariant under the Neumann moves from Figure \ref{fig:Neumann}.
We  use the restricted series in the following statement about the splitting move:

\begin{theorem}
\label{thm:Du2}
With assumptions as above, one has
\begin{align*}
\mathsf{Y}_{\tau_\circ}\left(\Gamma_\circ; q,t\right) 
=\,\,&
\sum_{w\in W}
\mathsf{Y}^w_{\tau_1}\left(\Gamma_1; q,t\right) \mathsf{Y}^w_{\tau_2}\left(\Gamma_2; q,t\right) \\
& \qquad 
\sum_{\alpha\in Q}
(-1)^{\ell(w)} k(\alpha)
\,q^{-\langle w(\rho), \rho + \alpha\rangle}
\,t^{d}
\end{align*}
where 
\begin{eqnarray*}
d&:=& -2\left(\xi_\circ(v_0)\right)^{-1}\xi_\circ(v_e)\,w(\rho)
-(2\rho+2\alpha)\\
&&- 2\left( \xi_1\left(v^1_*\right)\right)^{-1} \xi_\circ(v_e)\,w(\rho)
- 2\left( \xi_2\left(v^2_*\right)\right)^{-1} \xi_\circ(v_e)\,w(\rho).
\end{eqnarray*}
\end{theorem}

\begin{proof}
Let $x:=\xi_\circ(v_e)$. After replacing $w$ with $x^{-1}w$ and using the identities $\ell(x^{-1})=\ell(x)$ and $\langle x^{-1}w(\rho), 2\rho + 2\alpha\rangle = \langle w(\rho), x(2\rho + 2\alpha)\rangle$,
the statement is equivalent to
\begin{align}
\begin{split}
\label{eq:Du22}
\mathsf{Y}_{\tau_\circ}\left(\Gamma_\circ; q,t\right) 
=\,\,&
\sum_{w\in W}
\mathsf{Y}^w_{\tau_1}\left(\Gamma_1; q,t\right) \mathsf{Y}^w_{\tau_2}\left(\Gamma_2; q,t\right) \\
& \qquad 
\sum_{\alpha\in Q}
(-1)^{\ell(xw)} \, k(\alpha)
\,q^{-\langle w(\rho), x(\rho + \alpha)\rangle}
\,t^{d'}
\end{split}
\end{align}
where
\begin{align}
\begin{split}
\label{eq:d'}
d' :=&\,\, -2\left(\xi_\circ(v_0)\right)^{-1}w(\rho)
-(2\rho+2\alpha)\\
&- 2\left( \xi_1\left(v^1_*\right)\right)^{-1} w(\rho)
- 2\left( \xi_2\left(v^2_*\right)\right)^{-1} w(\rho).
\end{split}
\end{align}
We prove this version of the statement.

Let $B_\circ$ be the plumbing matrix for $\Gamma_\circ$ and $B_i$ for $\Gamma_i$ with $i\in\{1,2\}$.
A direct computation yields
\begin{align}
\begin{split}
\label{eq:sigmasplitting}
\pi\left( B_\circ\right) &= 1 + \pi\left( B_1\right) + \pi\left( B_2\right), \\
\sigma\left( B_\circ\right) &= \sigma\left( B_1\right) + \sigma\left( B_2\right),\\
\mathrm{tr}\left( B_\circ\right) &= e+ \mathrm{tr}\left( B_1\right) + \mathrm{tr}\left( B_2\right).
\end{split}
\end{align}
It follows that the factor in front of the sum in the series \eqref{eq:frontfactor} for $\Gamma_\circ$ has an extra factor 
with respect to the product of the analogous factors for $\Gamma_1$ and $\Gamma_2$
equal to
\begin{equation}
\label{eq:frontfactorRseries}
(-1)^{|\Delta^+|} \, q^{-\frac{1}{2}e \langle \rho, \rho \rangle}.
\end{equation}

For $i\in\{1,2\}$, the sum in the restricted $(q,t)$-series for $\Gamma_i$
is indexed by elements 
\[
\ell^i=\left( 2w(\rho), \ell^i_\sharp \right), 
\]
where $2w(\rho)$ is the entry corresponding to $v^i_{*}$, and $\ell^i_\sharp$ is the subtuple corresponding to the other vertices of $\Gamma_i$. 

For $w\in W$, $\alpha\in Q$, and
a pair of such $\ell^1$ and $\ell^2$, consider the tuple indexing the sum in the $(q,t)$-series for $\Gamma_\circ$ given by
\[
\ell^\circ =\left(-2w(\rho), -x(2\rho+2\alpha), 0, \ell^1_\sharp,  0, \ell^2_\sharp  \right).
\]
Here the  entry $-2w(\rho)$ corresponds to $v_0$, the entry $-x(2\rho+2\alpha)$ to $v_e$, the entries $0$ to $v^i_{*}$, and the subtuple $\ell^i_\sharp$ to the remaining vertices of $\Gamma_i$. Vice versa, the tuple $\ell^\circ$ uniquely determines $\left(w, \alpha, \ell^1, \ell^2\right)$, hence one has a bijection
\[
\ell^\circ \longleftrightarrow \left(w, \alpha, \ell^1, \ell^2 \right).
\]
In order to prove the statement, it suffices to verify that the contribution of $\ell^\circ$ to the left-hand side of \eqref{thm:Du2} equals the contribution of $\left(w, \alpha, \ell^1, \ell^2 \right)$ to its right-hand side. 
As the contribution to the  left-hand side of \eqref{thm:Du2} indexed by an element $\ell$ not of type $\ell^\circ$ vanishes, i.e., $c_{\Gamma_\circ, \xi_\circ} (\ell)=0$ in this case, this will conclude the proof.

\medskip

First, we verify the exponents of $q$.
Writing
\[
B^{-1}_1=\left(
\begin{array}{ccc}
a_{11} & \dots & a_{1n}\\
\vdots & \ddots & \vdots\\
a_{n 1} & \dots & a_{nn}
\end{array}
\right)
\quad\mbox{and}\quad
B^{-1}_2=\left(
\begin{array}{ccc}
b_{11} & \dots & b_{1m}\\
\vdots & \ddots & \vdots\\
b_{m 1} & \dots & b_{mm}
\end{array}
\right),
\]
a direct block matrix computation as in \cite[(5.4)]{ri2023refined} shows that 
\[
B^{-1}_\circ=\left(
\begin{array}{cccccccc}
a_{11}+b_{11}-e & 1 & -a_{11}& \dots & -a_{1n}& -b_{11}& \dots & -b_{1m}\\
1 & 0 & 0 & \dots &0&0&\dots& 0\\
-a_{1 1} & 0 & a_{11} & \dots & a_{1n} & 0 & \dots & 0\\
\vdots & \vdots & \vdots & \ddots & \vdots & \vdots & \ddots & \vdots\\
-a_{n 1} & 0 & a_{n1} & \dots & a_{nn} & 0 & \dots & 0\\
-b_{1 1} & 0 & 0 & \dots & 0& b_{11} & \dots & b_{1m} \\
\vdots & \vdots & \vdots & \ddots & \vdots & \vdots & \ddots & \vdots\\
-b_{m 1} & 0 & 0 & \dots & 0& b_{m1} & \dots & b_{mm} 
\end{array}
\right).
\]
This implies that
\[
\langle \ell^\circ, \ell^\circ\rangle = \langle \ell^1, \ell^1\rangle + \langle \ell^2, \ell^2\rangle -4e\, \langle \rho, \rho\rangle +4\langle w(\rho), x(2\rho+2\alpha)\rangle.
\]
After multiplying by $-\frac{1}{8}$ and considering the power of $q$ from \eqref{eq:frontfactorRseries}, it follows that the power of $q$ contributed by $\ell^\circ$ to the left-hand side of \eqref{eq:Du22} equals the power of $q$ contributed by $\left(w, \alpha, \ell^1, \ell^2 \right)$ to its right-hand side.

\medskip

Next, we verify the exponent of $t$. Starting from the left-hand side, the entries of $\ell^\circ$ corresponding to $v_0$ and $v_e$ contribute the following summand to the exponent of $t$:
\begin{align*}
&-2\left(\xi_\circ(v_0)\right)^{-1}w(\rho) -\left(\xi_\circ(v_e)\right)^{-1} x(2\rho+2\alpha) \\
&= -2\left(\xi_\circ(v_0)\right)^{-1}w(\rho) -(2\rho+2\alpha).
\end{align*}
Since the entries of $\ell^\circ$ corresponding to $v^1_*$ and $v^2_*$ are zero, they do not contribute to the exponent of $t$. 
On the right-hand side, since the entries of $\ell^1$ and $\ell^2$ corresponding to $v^1_*$ and $v^2_*$ are equal to $2w(\rho)$, these
 entries contribute the following summand to the exponent of $t$:
\[
2\left( \xi_1\left(v^1_*\right)\right)^{-1} w(\rho)
+2\left( \xi_2\left(v^2_*\right)\right)^{-1} w(\rho).
\]
Adding the contributions of the vertices of $\Gamma_i$ other than $v^i_*$ for $i\in\{1,2\}$ and taking into account the exponent $d'$ from \eqref{eq:d'}, 
it follows that the power of $t$ contributed by $\ell^\circ$ to the left-hand side of \eqref{eq:Du22} equals the power of $t$ contributed by $\left(w, \alpha, \ell^1, \ell^2 \right)$ to its right-hand side.

\medskip

Finally, we verify the equality of the coefficients. 
Start from the coefficient $c_{\Gamma_\circ, \xi_\circ} (\ell_\circ)$ computed as in \eqref{eq:cgammaxi}.
The entry of $\ell^\circ$ corresponding to $v_0$ contributes the factor $(-1)^{\ell(\iota w)}$. 
A simple consequence of the Weyl denominator formula yields 
\[
(-1)^{\ell(\iota w)}=(-1)^{|\Delta^+|}\,(-1)^{\ell(w)} \qquad \mbox{for $w\in W$}
\]
where $\iota$ is as in \eqref{eq:iota} --- this indeed follows by comparing the coefficients of $z^{2\iota w(\rho)}$ on the two sides of  \eqref{eq:Weyldenomformula}.
Moreover, since $K_{x,3}(z)$ is as in \eqref{eq:Ktwist}, the entry of $\ell^\circ$ corresponding to $v_e$ contributes the factor $(-1)^{\ell(x)}\, k(\alpha)$. Multiplying these two quantities, we obtain that the entries of $\ell^\circ$ corresponding to $v_0$ and $v_e$ contribute the factor
\[
(-1)^{|\Delta^+|}\,(-1)^{\ell(xw)} \, k(\alpha).
\]
The entries of $\ell^\circ$ corresponding to $v^1_*$ and $v^2_*$ are both zero and contribute factors of $1$. Correspondingly, the entry of $\ell^i$  equal to $2w(\rho)$ contributes the factor $(-1)^{\ell(w)}$ to $c_{\Gamma_i, \xi_i} (\ell^i)$. Multiplying by the entries corresponding to the vertices of $\Gamma_i$ other than $v^i_*$ for $i\in\{1,2\}$, yields 
\[
(-1)^{|\Delta^+|} \, c_{\Gamma_\circ, \xi_\circ} (\ell^\circ) = (-1)^{\ell(xw)} k(\alpha) \, c_{\Gamma_1, \xi_1} (\ell^1) \, c_{\Gamma_2, \xi_2} (\ell^2).
\]
Taking into account the factor contributed by \eqref{eq:frontfactorRseries}, it follows that 
 $\ell^\circ$ contributes the same $(q,t)$-monomial to the left-hand side of \eqref{eq:Du22} as  $\left(w, \alpha, \ell^1, \ell^2 \right)$ to its right-hand side. Hence the statement.
\end{proof}

\begin{remark}
\label{rmk:Du>=3}
When $u=1$, the Neumann move ($\mathrm{D}_{e,u}$) follows from the other five Neumann moves amongst trees. Indeed, when $e=0$, the move ($\mathrm{D}_{0,1}$) follows from move (C). When $e>0$, the move ($\mathrm{D}_{e,1}$) follows by applying the moves (A$-$), (B$-$), and then ($\mathrm{D}_{e-1,1}$). When $e<0$, the move ($\mathrm{D}_{e,1}$) follows by applying the moves (A$+$), (B$+$), and then ($\mathrm{D}_{e+1,1}$). Thus the move ($\mathrm{D}_{e,1}$) follows by induction on $e$.

Also, we remark that the case $u\geq 3$ follows by induction on the case $u=2$ and the other Neumann moves amongst trees. 
Hence the emphasis on the $u=2$ case here.
\end{remark}

%%%%%%%%%%%%%%%%%%%%%%%%%%%%%
%%%%%%%%%%%%%%%%%%%%%%%%%%%%%
%%%%%%%%%%%%%%%%%%%%%%%%%%%%%

\section{Examples}
\label{sec:ex}

\subsection{Lens spaces}
\label{sec:Lens}
Here we consider the case when $M$ is the lens space $L(p,1)$ for an integer $p\neq 0$.
We can assume that $\Gamma$ consists of a single vertex weighted by $p$. 
Hence $\spinc_Q\left(M \right) \cong \frac{2\,Q}{2p\,Q}$.
From Theorem \ref{thm:qtseriesinvarianceintroreduced}(ii), we can assume that the Weyl assignment is $\xi=(1_W)$. 

For $Q=A_1$ and $|p|\geq 3$, there are exactly three $\spinc$-structures that result in a non-zero $(q,t)$-series, namely:
\[
\mathsf{Y}_\tau\left(q,t\right) =
\left\{
\begin{array}{ll}
2\sigma\,q^{\frac{1}{4}(3\sigma-p)} & \mbox{for $a\equiv0$ mod $2p$,}\\[5pt]
 -\sigma\,q^{\frac{1}{4}(3\sigma-p)-\frac{1}{p}}\, t^2 & \mbox{for $a\equiv  2$ mod $2p$,}\\[5pt]
  -\sigma\,q^{\frac{1}{4}(3\sigma-p)-\frac{1}{p}}\, t^{-2} & \mbox{for $a\equiv - 2$ mod $2p$,}\\[5pt]
0   & \mbox{otherwise.}\\
\end{array}
\right.
\]
Here $\sigma=\mathrm{sign}(p)$.
For $|p|\leq 2$, some of the above congruence classes coincide, and thus the corresponding contributions add up.

\subsection{Brieskorn spheres}
\label{sec:Brieskorn}
Consider the case when $M$ is a Brieskorn homology sphere.
After reviewing the $q$-series for this $M$ and its independence on $\xi$, we give a closed formula for the $(q,t)$-series and show how the latter varies with~$\xi$. 

The plumbing tree $\Gamma$ can be assumed to be a star graph with a vertex of degree $3$ and with a definite plumbing matrix.
Thus for a tuple $\tau=\left(Q, a, \xi\right)$, the $q$-series $\mathsf{Y}_\tau\left(q\right)=\mathsf{Y}_\tau\left(q,1\right)$ exists. 
Since $M$ is a homology sphere, $a=\delta$ is the unique $\spinc$-structure on $M$.

Since $K_{x,n}(z)$ in \eqref{eq:Kxn} does not depend on $x$ for $n\in\{0,1,2\}$, it follows that 
for a tuple $\tau=\left(Q, a, \xi\right)$, the $q$-series $\mathsf{Y}_\tau\left(q\right)$ depends on $\xi$ at most up to the value of $\xi$ at the vertex of degree $3$.
Then applying Theorem \ref{thm:qtseriesinvarianceintroreduced}(ii), we deduce that for a tuple $\tau=\left(Q, a, \xi\right)$, the $q$-series $\mathsf{Y}_\tau\left(q\right)$ is in fact independent of $\xi$ and thus equals the series $\widehat{Z}(q)$ computed in this case for $Q=A_1$ in \cite{gukov2021two} and for arbitrary $Q$ in \cite{park2020higher}.

Specifically, select an order $v_0, v_1, v_2, v_3, \dots$ of the vertex set of $\Gamma$ so that $v_0$ is the vertex of degree $3$ and $v_1, v_2, v_3$ are the three vertices of degree $1$. A computation reviewed in \cite[\S 6.1]{MT1} shows that
\[
\mathsf{Y}_\tau\left(q\right) =
q^{-\frac{1}{2}(3s+\mathrm{tr}\,B)\langle \rho, \rho \rangle} 
\mathop{\sum_{\gamma\in Q}}_{w_1, w_2, w_3\in W}  (-1)^{\ell(w_1w_2w_3)} d\left(\gamma\right)q^{-\frac{1}{8}\langle f, f\rangle}
\]
where 
\begin{align*}
f=\left(\gamma, 2w_1(\rho), 2w_2(\rho), 2w_3(\rho), 0,\dots,0 \right)\in Q^s \\
\mbox{with $\gamma\in 2\rho+2Q$ and $w_1,w_2,w_3\in W$}
\end{align*}
and 
\[
d\left(\gamma\right) := k\left( -\frac{1}{2}\,\gamma -\rho\right)
\]
with $k(\,\,)$ equal to the Kostant partition function as in \eqref{eq:Kostant}.

Next, we show how this computation can be refined to include the variable~$t$ and how on the contrary the resulting $(q,t)$-series varies with $\xi$. 
For a Weyl assignment $\xi=(\xi_0, \xi_1, \xi_2, \xi_3, 1_W, \dots, 1_W)$, the exponent of $t$ as in \eqref{eq:exp_t} is 
\[
e(\xi, f):=\xi_0^{-1}(\gamma) + 2 \,\xi_1^{-1} w_1(\rho) + 2 \,\xi_2^{-1} w_2(\rho) + 2\,\xi_3^{-1} w_3(\rho).
\]
Then the $(q,t)$-series is
\[
\mathsf{Y}_\tau\left(q,t\right) =
q^{-\frac{1}{2}(3s+\mathrm{tr}\,B)\langle \rho, \rho \rangle} 
\mathop{\sum_{\gamma\in Q}}_{w_1, w_2, w_3\in W}  (-1)^{\ell(w_1w_2w_3)} d\left(\gamma\right)\, t^{e(\xi, f)}\,q^{-\frac{1}{8}\langle f, f\rangle}.
\]
E.g., when $\Gamma$ is negative definite, $\xi=(1_W, \dots, 1_W)$ satisfies the conditions of a Weyl assignment. For this choice,
the series $\mathsf{Y}_\tau\left(q,t\right)$ coincides with the series $\Zdhat\left(q,t^2\right)$ from \cite{akhmechet2023lattice},  computed for Brieskorn spheres in \cite{liles2023infinite}. 

Since $a$ is the unique $\spinc$-structure, Theorem \ref{thm:qtseriesinvarianceintroreduced}(ii) implies that 
two Weyl assignments $\xi$ and $\xi'$ in the same orbit by the $W$-action (i.e., $\xi'=w(\xi)$ for some $w\in W$) yield the same series $\mathsf{Y}_\tau\left(q,t\right)$. However, when $\xi$ and $\xi'$ are not in the same orbit by the $W$-action, the resulting series $\mathsf{Y}_\tau\left(q,t\right)$ are in general distinct, although equal at $t=1$.

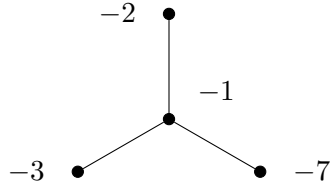
\begin{figure}[t]
\[
\begin{tikzpicture}[baseline={([yshift=0ex]current bounding box.center)}]
      \path(0,0) ellipse (2 and 2);
      \tikzstyle{level 1}=[counterclockwise from=90, level distance=40mm, sibling angle=120]
      \node [draw, circle, fill, inner sep=1.5, label={[label distance=10]30:$-1$}] (A0) at (0:0) {}
            child {node [draw, circle, fill, inner sep=1.5, label={[label distance=10]180:$-2$}]{}}
            child {node [draw, circle, fill, inner sep=1.5, label={[label distance=10]180:$-3$}]{}}
	    child {node [draw, circle, fill, inner sep=1.5, label={[label distance=10]0:$-7$}]{}};
    \end{tikzpicture}
\]
\caption{The Brieskorn sphere $\Sigma(2,3,7)$.}
\label{fig:S237}
\end{figure}

\subsection{The Brieskorn sphere $\Sigma(2,3,7)$}
As a special case of \S\ref{sec:Brieskorn}, consider $M=\Sigma(2,3,7)$.
This is obtained from the negative-definite plumbing tree in Figure \ref{fig:S237}.
For $Q=A_1$ and $\xi=(1_W, 1_W, 1_W, 1_W)$, one has
\begin{eqnarray*}
\mathsf{Y}_\tau\left(q,t\right) &=& 
q^{1/2}\left(t^2 -q -q^5 +q^{10}\,t^{-2} -q^{11} +q^{18}\,t^{-2} +q^{30}\,t^{-2} -q^{41}\,t^{-4} \right.\\
&&
\left. \,\,\qquad + q^{43} -q^{56}\,t^{-2} -q^{76}\,t^{-2} +q^{93}\,t^{-4}
+  O(q^{96})\right) .
\end{eqnarray*}
Here $O(q^x)$ stands for $q^x$ times a series in non-negative powers of $q$.
This series matches the computations of $\Zdhat\left(q,t^2\right)$ in \cite{liles2023infinite}.

Still for $Q=A_1$, the input $\xi=(\iota, \iota, \iota, \iota)$ yields the same series, as per Theorem \ref{thm:qtseriesinvarianceintroreduced}(ii). However, the input $\xi=(\iota, 1_W, 1_W, 1_W)$ yields a distinct series:
\begin{eqnarray*}
\mathsf{Y}_\tau\left(q,t\right) &=& 
q^{1/2}\left(t^{-4} 
-q\,t^{-2} 
-q^5\,t^{-2} 
 +q^{10} 
 -q^{11}\,t^{-2} 
  +q^{18} 
  +q^{30} 
  -q^{41}\,t^{2} \right.\\
&&
\left. \,\,\qquad + q^{43}\,t^{-6} 
 -q^{56}\,t^{-4} 
 -q^{76}\,t^{-4} 
 +q^{93}\,t^{2}
+  O(q^{96})\right) .
\end{eqnarray*}
Overall, there are $8$ Weyl assignments in different orbits of $W=\mathbb{S}_{2}$ that yield $8$ distinct series.

For $Q=A_2$ and $\xi=(1_W, \dots, 1_W)$, one has
\begin{eqnarray*}
\mathsf{Y}_\tau\left(q,t\right) &=& 
q^2 \, t_1^4 \, t_2^4 
-q^3 \left( t_1^4\, t_2^2 + t_ 1^2\, t_2^4\right)
+q^5 \left( t_1^2 + t_2^2\right)
+q^6 (2\, t_1^2 \, t_2^2 -1)\\
&&
-q^7 \left( t_1^4\, t_2^2 + t_ 1^2\, t_2^4\right)
-2q^{10} 
+4q^{11}\, t_1^2 \, t_2^2 
+q^{12} \left( t_1^4 + t_2^4\right)\\
&&
-q^{13} \left( t_1^2 + t_2^2 + t_1^4\, t_2^2 + t_ 1^2\, t_2^4 \right)
+q^{15}\left( t_1^{-2} + t_2^{-2}\right)\\
&&
-q^{16}\left(2 + t_1^2 + t_2^2 +t_1^2\,t_2^{-2} + t_1^{-2}\,t_2^2 \right)
+O(q^{17}).
\end{eqnarray*}
Specializing at $t_1=t_2=1$ recovers the $q$-series from \cite{park2020higher}.

%%%%%%%%%%%%%%%%%%%%%%%%%%%%
%%%%%%%%%%%%%%%%%%%%%%%%%%%%
%%%%%%%%%%%%%%%%%%%%%%%%%%%%

\section*{Acknowledgments} 

AHM was partially supported by  NSF award DMS-2204148. 
NT was partially support by NSF award DMS-2404896 and a Simons Foundation's Travel Support for Mathematicians gift.
The authors thank the referee for their careful reading and detailed comments.

%%%%%%%%%%%%%%%%%%%%%%%%%%%%%%%%%%%%%%
\bibliographystyle{alpha}
\bibliography{Biblio}
%%%%%%%%%%%%%%%%%%%%%%%%%%%%%%%%%%%%%%

\end{document}